\documentclass[10pt]{amsart}
\usepackage{amsmath,amssymb,hyperref}
\usepackage{cite,graphicx}
\usepackage{xcolor}
\usepackage{verbatim} 
\usepackage{enumerate}
\usepackage[noend]{algpseudocode}
\usepackage{algorithm}

\title{Linear and Rational Factorization of Tropical Polynomials}
\author{Bo Lin}
\address{Department of Mathematics, University of Texas at Austin, TX 78712, USA, and School of Mathematics, Georgia Institute of Technology, GA 30332, USA}
\email{linbomath@gmail.com}
\author{Ngoc Mai Tran}
\address{Department of Mathematics, University of Texas at Austin, TX 78712, USA, and the Hausdorff Center for Mathematics, Bonn 53115, Germany}
\email{ntran@math.utexas.edu}

\theoremstyle{plain}
\newtheorem{theorem}{Theorem}
\newtheorem{lemma}[theorem]{Lemma}
\newtheorem{corollary}[theorem]{Corollary}
\newtheorem{conjecture}[theorem]{Conjecture}
\newtheorem{proposition}[theorem]{Proposition}

\theoremstyle{definition}
\newtheorem{remark}[theorem]{Remark}
\newtheorem{definition}[theorem]{Definition}
\newtheorem{example}[theorem]{Example}
\newtheorem{open}[theorem]{Open question}

\newcommand{\N}{\mathbb{N}}
\newcommand{\R}{\mathbb{R}}
\newcommand{\Z}{\mathbb{Z}}
\newcommand{\B}{\mathcal{B}}

\newcommand{\T}{\mathcal{T}}
\renewcommand{\S}{\mathcal{S}}
\DeclareMathOperator{\spanv}{span}

\DeclareMathOperator{\NS}{\mathbb{N}[\mathcal{S}]}
\DeclareMathOperator{\ZS}{\mathbb{Z}[\mathcal{S}]}
\DeclareMathOperator{\ES}{\mathbb{E}[\mathcal{S}]}

\DeclareMathOperator{\face}{face}

\DeclareMathOperator{\argmax}{argmax}

\DeclareMathOperator{\conv}{conv}
\DeclareMathOperator{\Ima}{Im}
\DeclareMathOperator{\newt}{Newt}

\DeclareMathOperator{\Newt}{Newt}
\begin{document}
\begin{abstract}
Already for bivariate tropical polynomials, factorization is an NP-Complete problem. 
In this paper, we give an efficient algorithm for factorization and rational factorization of a rich class of tropical polynomials in $n$ variables. Special families of these polynomials have appeared in economics, discrete convex analysis, and combinatorics.
Our theorems rely on an intrinsic characterization of regular mixed subdivisions of integral polytopes, and lead to many open problems of interest in discrete geometry.
\smallskip
\noindent \textbf{Keywords.} Tropical polynomials, factorization, rational factorization, $M$-convex, $L$-convex, generalized permutohedra, Legendre transform, Minkowski sum \\
\textbf{2010 MSC.} 13P05, 14T05, 52B20
\end{abstract}
\maketitle

Consider the max-plus tropical algebra $(\R, \odot, \oplus)$, $a \odot b = a + b$, $a \oplus b = \max(a,b)$. A unit $f$ is a tropical polynomial inducing trivial regular subdivision of its Newton polytope $\newt(f)$. For a set of lattice polytopes $\S$ in $\R^n$, an $\S$-unit $f$ is a unit where $\newt(f)$ is a translation of some polytope in $\S$. We say that an $n$-variate polynomial $f$ is $\S$-factorizable if it equals a product of $\S$-units. The set of $\S$-factorizable polynomials $\NS$ is a monoid. Deciding if a given polynomial $f$ is $\S$-factorizable is an instance of the classic factorization problem in tropical geometry~\cite{Maclagan15,Speyer09}, which remains open except for univariate polynomials \cite{grigg2007factorization,kim2005factorization}. 

By the Cayley trick \cite{Sturmfels94}, this problem is equivalent to the following question on regular mixed subdivisions. Fix a set of lattice polytopes $\S$ in $\Z^n$, and let $\Delta$ be a regular subdivision of a certain Newton polytope $P$. Is $\Delta$ a mixed regular subdivision with respect to some sequence of polytopes in $\S$? Computing and enumerating regular mixed subdivisions is a central problem polyhedral geometry \cite{detriangulations}, however, this problem too seems open. There are two layers of difficulties here. First, consider the special cases where $\Delta$ is a trivial subdivision. Then $\Delta$ is mixed if and only if $P$ can be decomposed as the Minkowski sum of some sequence of polytopes in $\S$. For $n \geq 2$, Minkowski factorization of a polytope is an instance of subset sum, and thus is NP-Complete \cite{gao2001decomposition,Tiwary2008}. However, there are many algorithms one could apply and efficient shortcuts in special cases \cite{fukuda2004zonotope,fukuda2005computing,gritzmann1993minkowski,mount1991combinatorial,weibel2007minkowski}. 
Now, consider the general case, when $\Delta$ is not a trivial subdivision. Suppose that we have an oracle that can not only decide if $P$ can be written as such a Minkowski sum, but also give us the decomposition when it is possible. A necessary condition for $\Delta$ to be mixed is that each cell of $\Delta$ must be factorizable as a Minkowski sum of faces of some sequence polytopes in $\S$. We stress that this sequence could be non-unique, and it could differ from cell to cell. In this case, we say that $\Delta$ is \emph{locally factorizable}. Local factorization can be checked using the oracle. However, this is not sufficient to imply that $\Delta$ is a regular mixed subdivision in general. Being a mixed subdivision is a global condition: it requires that the cells must have a consistent mixed labels, in other words, the same sequence of polytopes and their facets must factorize all of the cells in $\Delta$. 

This paper gives a large class of polytopes $\S$ such that if $\Delta$ is locally factorizable, then it is a mixed subdivision to a unique sequence of polytopes in $\S$, up to permutation. With this property, deciding if a given polynomial is $\S$-factorizable reduces to a series of independent Minkowski decomposition problems. In particular, our result supplies a computational shortcut to verifying regular mixed subdivisions. 
 
\begin{theorem}[Local Factorization]\label{thm:main}
Let $\S$ be a set of lattice polytopes in $\R^n$. If $\S$ is a positive basis, then $\NS$ has unique and local factorization. In other words, if $f$ is a tropical polynomial such that each cell in its regular subdivision $\Delta_f$ is a Minkowski sum of some polytopes in $\S$, then $f$ admits a unique minimal factorization into a tropical product of $\S$-unit polynomials.
\end{theorem}

The name `positive basis' comes from the fact that $\mathcal{S}$ is associated with a basis of a certain vector space, with an extra orientability condition (cf. Definition \ref{def:positive.basis}). Positive bases are simple to construct and verify (cf. Section \ref{sec:algorithm}). As a result, Theorem \ref{thm:main} applies to a wide class of polynomials, with a number of interesting consequences. 

The second main result of our paper considers the problem of rational factorization.  For $n = 1$, rational factorization of tropical polynomials have been studied in \cite{tsai2012working, Korhonen15}, in connections to tropical meromorphic functions. Rational tropical polynomials form a rich class of functions that have appeared in a variety of applications: unirational varieties \cite{Draisma13}, ultra discrete equations \cite{Korhonen15}, auction theory \cite{BaldwinSlides} and topological data analysis \cite{Verovsek16}. They are equivalent to deep neural networks with ReLU activation functions and integer weight coefficients \cite{zhang2018tropical}. Unit tropical polynomials correspond to neural networks with linear decision functions, the simplest class. Rationally factorizable tropical polynomials is the subset of neural networks that can be expressed as the \emph{sum} of independent linear decision functions. To compute the rational factorization of a tropical polynomial is to transform a neural network into a much simpler description. Knowing how and when we can do this is an important step towards understanding deep neural networks. 

Formally, say that $f$ is rationally factorizable with respect to $\S$, or $\S$-rational, if $f \odot g \in \NS$ for some $g \in \mathbb{N}[\S']$ for some set of lattice polytopes $\S'$. Say that it is strong $\S$-rational if $f \odot g \in \NS$ for some $g \in \NS$. The set of $\S$-rationals $\ES$ and the set of strong $\S$-rationals $\ZS$ are both monoids, and they are much richer than $\NS$. 
In terms of cells of $\Delta_f$, the extension from factorization to rational factorization is the extension from Minkowski sums to signed Minkowski sums, a useful operation in combinatorial geometry \cite{Ardila10,gritzmann1999algorithmic}. For $f$ to be $\S$-rational (resp. strongly $\S$-rational), a necessary but not sufficient condition is that each cell of $\Delta_f$ can be expressed a signed Minkowski sum where all of the positive terms (resp. all of the terms) lie in $\S$.
The following theorem states that if $\S$ is a positive basis, then this local condition is sufficient for rational factorization. 

\begin{theorem}[Local rational factorization]\label{thm:SS}
Suppose $\S$ is a positive basis. Then $\ZS$ has unique and local factorization. 
\end{theorem}

The third main result of our paper gives a sufficient condition on $\S$ for $\ES$ to have unique and local factorization. We provide a large and important family of $\S$ with this condition, and thereby generalizes a result in auction theory of Baldwin, Golberg and Klemperer \cite{BaldwinSlides}.

\begin{theorem}[Local strong rational factorization]\label{thm:ES}
Suppose $\S$ is a positive basis. Then $\ZS = \ES$ if and only if $\S$ is full. In this case, 
$f \in \ZS$ if and only if the edges of $\Delta_f$ are parallel to integer multiples of edges ($1$-dimensional polytopes) in $\S$, as vectors.
\end{theorem}

Loosely speaking, being full means $\S$ is a maximal basis amongst all bases that have the same set of primitive edges (cf. Definition \ref{defn:basis}). 
One important example is the set $\S_{\mathbb{G}}$ which consists of all standard simplices indexed by cliques of a given graph $\mathbb{G}$ on $n$ nodes (cf. Proposition \ref{prop:graphical}). For the complete graph $K_n$, for instance, $\S_{K_n}$ consists of the standard simplex in $\R^n$ and its faces. Tropical polynomials in $\mathbb{N}[\S_{K_n}]$
define tropical hyperplane arrangements. These play an important role in defining tropical polytopes and their generalizations \cite{develin2004tropical, fink2015stiefel, joswig2016weighted}, and have applications in economics and combinatorics \cite{allamigeon2015tropicalizing,ardila2009tropical,Baldwin13,crowell2016tropical,joswig2016cayley,shiozawa2015international}. 
In discrete convex analysis, the set $\mathbb{E}[\S_{K_n}]$ is the set of $L$-convex functions whose domain is all of $\Z^n$ \cite{Murota03}. Their Legendre transforms are $M$-convex functions with compact domains. These functions feature prominently in polymatroid theory and have many interesting properties and applications, see the monographs \cite{Fujishige05,Murota03} and references therein. 

Theorem \ref{thm:ES} generalizes several statements known in the literature. The case for unit polynomials corresponds to the signed Minkowski decomposition of generalized permutohedra into standard simplices \cite{postnikov2008faces,Postnikov09,frank2014characterizing}. There are many papers devoted to their combinatorics and applications \cite{Ardila10,castillo2015ehrhart,doker2011geometry,mohammadi2016generalized,morton2009convex,postnikov2008faces,Postnikov09,postnikov2013poset}.
 Baldwin, Golberg and Klemperer \cite{BaldwinSlides} showed that $\mathbb{Z}[\S_{K_n}] = \mathbb{E}[\S_{K_n}]$ and gave an algorithm to produce a certificate of factorization. Since this case is of particular interest to auction theory, which may utilize different notations and languages, for ease of reference we restate their theorem in full here.
\begin{theorem}[\cite{BaldwinSlides}]\label{thm:baldwin}
Let $f$  be a tropical polynomial. There exists $g$ a product of linear polynomials such that $f \odot g$ is a product of linear polynomials if and only if the edges in $\Delta_f$ are parallel to $e_i - e_j$ for $i,j \in \{0,1,\ldots,n\}$, with the convention that $e_0$ is the origin.
\end{theorem}
Theorem \ref{thm:SS} strengthens the above results in two directions. First, one has unique factorization. Second, only edges that appear in $\Delta_f$ may appear in both the numerator and denominator of the rational factorization. 
\begin{theorem}\label{prop:baldwin}
Let $f$ be a tropical polynomial in $n+1$ variables such that the edges of $\Delta_f$ are parallel to $e_i - e_j$ for $i,j \in \{0,1,\ldots,n\}$. Let $\mathbb{G}(f)$ be the graph on $n+1$ nodes, where $(i,j) \in\mathbb{G}(f)$ whenever there exists an edge in $\Delta_f$ parallel to $e_i - e_j$. Then $f \in \Z[\S_{\mathbb{G}(f)}]$. That is, there is a unique way to write $f$ as a tropical rational function, the denominator and numerator are product of linear polynomials, such that no new edges are introduced. 
\end{theorem}

Our theorems are constructive. Given a set $\S$ of lattice polytopes in $\Z^n$, Algorithm~\ref{alg:positive.basis} certifies if it is a positive basis, and if it is, for any $f \in \ZS$, Algorithm~\ref{alg:gen-ZS} produces the unique minimal $g \in \NS$ such that $f \odot g \in \NS$, and Algorithm~\ref{alg:gen-NS} produces the unique factorization of a polynomial in $\NS$. Our algorithms have polynomial run time with respect to the number of polytopes in~$\S$, however, this generally scales exponential in $n$. We implement these algorithms with the softwares \textsf{Maple} and \textsf{polymake} \cite{polymake}. Codes for the examples in this paper can be found at \url{https://github.com/linbomath/TropPolyFactor}.

\textbf{Organization}. We collect background materials and discuss subtleties surrounding factorization in Section \ref{sec:background}. We discuss factorization of units in Section~\ref{sec:unit}, define positive bases and prove the main results in Section \ref{sec:main}. Section \ref{sec:positive.basis} introduces two families of full positive bases and some examples. Section \ref{sec:algorithm} and \ref{sec:example} show the various algorithms and their outputs on numerical examples. We conclude with open problems in Section~\ref{sec:summary}.

\textbf{Notations}.
For a set of vectors $B \in \R^n$, write $\mathbb{N}B$ for their span over $\mathbb{N}$, $\mathbb{Z}B$ for their span over $\mathbb{Z}$. Say that $B'$ is a basis of $B$ if $\mathbb{Z}B = \mathbb{Z}B'$, and the vectors in $B'$ are linearly independent over $\mathbb{Z}$. For polytopes $P,Q \subset \R^n$, $c \in \mathbb{N}$, let $P+Q$ denote their Minkowski sum, $c \cdot P$ denote the dilation of $P$. Say that $P$ is equivalent to~$Q$, written $P \equiv Q$, if $P = v + Q$ for some $v \in \Z^n$. If there exists a polytope $S \subset \R^n$ such that $Q+S = P$, say that $Q$ is a Minkowski summand of $P$, and write $Q \leq P$. Let $\mathcal{N}(P)$ denote the normal fan of a polytope $P$. A face of $P$ is either $P$ itself, or any polytope obtained as the set of maximizers of some linear functional over $P$. A proper face of $P$ is a face of $P$ that is neither $P$ nor one of its vertices. We denote the face of $P$ supported by a vector $v$ by $\face_v(P)$. For vectors $v,w \in \R^n$, write $v \cdot w$ for their inner product. For a matrix $H \in \Z^{r \times n}$, write $\Ima_\Z(H) := \{Hv: v \in \Z^n\}\subset \Z^r$ for the image of $\Z^n$ under $H$.

\section{Background}\label{sec:background}

\subsection{Background on tropical polynomials}
A tropical polynomial in $n$ variables is a piecewise linear, convex function $f: \R^n \to \R$ such that there exists $c_a \in \R, a \in A \subset \R^n$ where
\begin{equation}\label{eqn:f}
f(x) = \bigoplus_{a \in A} \left(c_a \odot x^{\odot a}\right) = \max_{a \in A} \left(c_a + \sum_{i=1}^na_ix_i\right) \mbox{ for all } x \in \R^n.
\end{equation}
The convex hull of $A$ is called the Newton polytope of $f$, denoted as $\newt(f)$. Points $a \in A$ are said to be lifted by the height function $a \mapsto c_a$. The Legendre transform of $f$ is the function $f^\ast: \R^n \to \R \cup \{+\infty\}$, given by
$$ f^\ast(y) = \sup_{x \in \R^n} \left(\sum_{i=1}^ny_ix_i - f(x)\right) \mbox{ for all } y \in \R^n. $$
The Legendre transform $f^\ast$ has a particularly simple interpretation: $f^\ast(y) \neq +\infty$ if and only if $y$ is in $\newt(f)$, and on this set, the graph of $f^\ast$ equals the lower convex hull of the points $\{(a,-c_a): a \in A\}$. The projection of this graph onto $\newt(f)$ is called the regular subdivision of $\newt(f)$ induced by $f$, denoted as $\Delta_f$. 
A regular subdivision $\Delta_f$ is called mixed with respect to a sequence of polytopes $(F_1, F_2, \ldots, F_r)$ if each 
cell in $\Delta_f$ equals to a Minkowski sum $\sum_{i=1}^r B_i$, where $B_i$ is a face of $F_i$ for each~$i$, and such that this representation intersects properly as a sum, meaning that if $\sigma = \sum_{i=1}^r B_i$, and $\sigma' = \sum_{i=1}^r B_i'$ for faces $B_i, B_i'$ of $F_i$, then the intersection of $B_i$ and $B'_i$ is a face of both, for each $i \in \{1,\ldots,r\}$. The polyhedral version of the Cayley trick \cite{Sturmfels94} can be restated in the language of tropical factorization as follows.
\begin{theorem}[Cayley trick]
Let $\mathcal{S}$ be a set of polytopes. Then $f \in \NS$ if and only if 
$\Delta_f$ is a regular mixed subdivision of $\newt(f)$ with respect to a sequence of possibly repeated polytopes in $\S$.
\end{theorem}

The tropical hypersurface $\T(f)$ defined by $f$ is the set of points in $\R^n$ where the graph of $f$ are not differentiable. A tropical hypersurface defines a balanced, weighted polyhedral complex, pure of dimension $n-1$ in $\R^n$, and a converse of this statement also holds, see \cite[Proposition 3.3.10]{Maclagan15}. A straight-forward definition chase from this result gives Corollary \ref{cor:factorize}, a  characterization of when a unit can be factorized off a given tropical polynomial.

\begin{definition}\label{defn:neighbor}
Let $\sigma,\sigma'$ be maximal cells in $\Delta_f$. Say that $\sigma'$ is a neighbor of $\sigma$ in direction $v$ if $\face_v(\sigma)$ is a maximal face of $\sigma$, and $\face_v(\sigma) = \face_{-v}(\sigma')$. Say that $\sigma'$ is in direction $v$ from $\sigma$ if there exists a sequence of cells $\sigma_1, \sigma_2, \ldots, \sigma_k$ in $\Delta_f$ where $\sigma_1 = \sigma$, $\sigma_k = \sigma'$, and $\sigma_{i+1}$ is a neighbor of $\sigma_i$ in direction $v$, for $i = 1, \ldots, k-1$.
\end{definition}

\begin{corollary}\label{cor:factorize}
Let $h$ be a tropical polynomial. Then 
	$\T(h) = \T(f) \cup \T(h')$ for some unit $f$ with $\newt(f) = F$ and some tropical polynomial $h'$ if and only if there exists a cell $\sigma \in \T(h)$ where $F \leq \sigma$, and for each maximal face $\face_v(\sigma)$ of $\sigma$, all cells of $\Delta_h$ in direction $v$ from $\sigma$ has $\face_v(F)$ as a Minkowski summand. 
\end{corollary}

\subsection{What counts as factorization}
There are at least three natural notions of `equality' for two tropical polynomials $f$ and $g$ in $n$ variables $x_1, \ldots, x_n$. 
\begin{enumerate}
  \item As algebraic polynomials: $f =_{1} g$ if and only if $c_a(f) = c_a(g)$ for all $a \in \Z^n$. 
  \item As functions: $f =_{2} g$ if and only if $f(x) = g(x)$ for all $x \in \R^n$.
  \item As balanced polyhedral complexes: $f =_{3} g$ if and only if $\T(f) = \T(g)$ as sets and as balanced weighted polyhedral complexes. 
\end{enumerate}
One can check that $f =_1 g$ implies $f =_2 g$, and $f =_2 g $ implies $f =_3 g$, but the converses are not true. Often equality as functions is taken to be the definition of equality in factorization problems \cite{Maclagan15, Speyer09}. As with classical factorization, it is more natural to consider this equality up to multiplication by constants and monomials. On the surface this seems to be a fourth notion of equality. However, we show that this is exactly $=_3$, and this is the notion of equality for tropical polynomials used throughout this paper. 

\begin{lemma}
We have $f =_{3} g \Leftrightarrow f =_{2} a\odot x^{\odot v} \odot g$ for some $v \in \Z^n, a \in \R$. 
\end{lemma}
\begin{proof}
Suppose $f =_{2} a\odot x^{\odot v} \odot g$. Then $\T(f) = \T(g)$. Conversely, suppose $\T(f) = \T(g)$. The weighted polyhedral complex $\T(f)$ uniquely determines $\Delta_f$ up to a translation, thus, 
$\Delta_f = \Delta_g +v$ for some $v \in \Z^n$. Furthermore, a face $\sigma$ lifted in the graph of $f^\ast$ is supported by the same set of vectors as the face $\sigma+v$ lifted in the graph of $g^\ast$. Thus $f^\ast(x) = g^\ast(x) + a$ for some $a \in \R$. Since the function $f$ is uniquely determined by its Legendre transform, rewriting in polynomial terms gives $f =_{2} a\odot x^{\odot v} \odot g$.
\end{proof}

\begin{remark}
In several papers \cite{grigg2007elementary, izhakian2008tropical,tsai2012working} one associates a function $f$ with the unique polynomial $\bar{f}$ where $f =_{2} \bar{f}$ and all lattice points in $\Delta_{\bar{f}}$ are lifted, including interior points. This gives the stronger equivalence between $=_3$ and $=_1$
$$f =_{3} g \Leftrightarrow \bar{f} =_{1} a\odot x^{\odot v} \odot \bar{g}.$$ 
We do not take this approach here, as tropical multiplication does not commute with taking this canonical element. That is, for general polynomials $f$ and $g$,
\begin{equation}\label{eqn:non.commute}
\bar{f} \odot \bar{g} \neq_{1} \overline{f \odot g} =_{1} \overline{\bar{f}\odot \bar{g}}.
\end{equation}
Characterizing when $\bar{f} \odot \bar{g} =_{1} \overline{f \odot g}$ is the problem of finding competitive equilibrium in the product-mix auctions pioneered in \cite{Baldwin13}. For connections to integer programming and toric geometry, see \cite{Tran15}.
\end{remark}

\subsection{Signed Minkowski sums as vector additions}

\begin{definition}\label{defn:mindiff}
For non-empty polytopes $P,Q \subset \R^n$, if there exists another nonempty polytope $R\subset \R^n$ such that $P=Q+R$, then the signed Minkowski sum $P-Q$ is defined as $R$. According to Lemma \ref{lem:gooddiff}, such $R$ must be unique, so it is well-defined. 
\end{definition}

\begin{lemma}\label{lem:gooddiff}
	For non-empty polytopes $P,Q,R \subset \R^n$, if $P=Q+R$, then $R = \{x\in \R^n \mid x+Q\subset P \}$.
\end{lemma}

\begin{proof}
	For any $x\in R$ and $q\in Q$, we have $x+q\in P$, hence $x+Q\subset P$. Conversely, suppose there exists a point $y\in \R^{n}$ such that $y+Q\subset P$ while $y\notin R$. By Farkas Lemma, there exists a hyperplane such that $y$ and $R$ are separated by it. In other words, there exists a linear function $l$ defined on $\R^{n}$ such that $l(y)>0$ while $l(r)<0$ for all $r\in R$. Since $Q$ is closed, $l$ attains its maximum on $Q$ at some point $q1$. Then $l(y+q1)=l(y)+l(q1)>l(q1)$. Note that $y+q1\in P=Q+R$, so there exists $r\in R$ and $q2\in Q$ such that $y+q1=r+q2$. But $l(r+q2)=l(r)+l(q2)<0+l(q1)=l(q1)$, a contradiction! Hence such point $y$ does not exist and Lemma \ref{lem:gooddiff} is proved.
\end{proof}

In general, $P+(-Q) \neq P-Q$. For instance, if $P=Q$ and $P$ is a symmetric polytope around the origin, so $P=-P$, then $P+(-P) = 2 \cdot P$. In contrast, $P-P = \{\mathbf{0}\}$. 

\begin{lemma}[Signed Minkowski sum operations]\label{lem:signed.minkowski}
Let $P,Q,R,S \subset \R^n$ be non-empty lattice polytopes. 
\begin{enumerate}[(i)]
  \item If $P-Q$ is well-defined, then $(P-Q)+Q = P$.
  \item If $P-Q$ and $(P-Q)-R$ are well-defined, then so is $P - (Q+R)$ and it equals to $(P-Q) - R$.
  \item If $P-Q$ is well-defined, then so is $(P+R)-Q$ and it equals to $(P-Q)+R$.
  \item If both $P-Q$ and $R-S$ are well-defined, then so is $(P+R) - (Q+S)$ and it equals to $(P-Q)+(R-S)$.
  \item If $P-Q$ is well-defined, then it is a convex lattice polytope. 
\end{enumerate}
\end{lemma}

\begin{proof}
Statements (i), (ii), (iii) and (iv) follow directly from Definition \ref{defn:mindiff}. Statements (v) can be found in \cite[Lemma 11.1]{Postnikov09}. For statement (v), $P-Q$ being lattice appears in \cite[Lemma 11.1]{Postnikov09}. As for being a lattice polytope, let $S = P-Q$. Let $V(S), V(P), V(Q)$ be the set of vertices of $S, P$ and $Q$, respectively. Note that $V(P) \subseteq V(Q) + V(S)$. Take a vertex $s \in V(S)$. If there exists a vertex $q \in V(Q)$ with $q+s \in V(P)$, then $s \in \Z^n$. If there is no such vertex $q \in v(Q)$, then $s + Q \subset P \backslash V(P)$. Since $s+Q$ is a closed polytope, there exists a direction $w \in \R^n$ and a small $\epsilon > 0$ such that $s + [-\epsilon,\epsilon] \cdot w + Q \subset P$. So $s$ cannot be a vertex of $S$, a contradiction. Therefore, all vertices of $S$ are in $\Z^n$, as claimed.
\end{proof}
 
\begin{definition}\label{defn:signed.minkowski}
	Let $P_{1},P_{2},\cdots,P_{m}$ be non-empty polytopes in $\R^{n}$, $c_{1},c_{2},\cdots,c_{m}\in \Z$ with at least one being positive. If there exists a polytope $P'$ such that
	
	\begin{equation}
		\sum_{c_{i}<0}{(-c_{i})P_{i}}+P' = \sum_{c_{i}>0}{c_{i}P_{i}},
	\end{equation}
	
	then the signed Minkowski sum $\sum_{i=1}^{m}{c_{i}P_{i}}$ is defined to be $P'$. Throughout this work, when we write
	$\sum_{i=1}^{m}{c_{i}P_{i}}$, we mean the signed Minkowski sum. 
\end{definition} 

\begin{remark}\label{rem:signed.minkowski}
	\begin{enumerate}
		\item By Lemma \ref{lem:signed.minkowski}(i)-(iv), a signed Minkowski sum is independent of the order of its summands.
		\item For general polytopes $P,Q$, $P-Q$ is defined as $\{x\in \R^n \mid x+Q\subset P\}$ \cite{Postnikov09}. However, for this definition $P-Q$ could be empty and Lemma \ref{lem:signed.minkowski}(iii) may not hold. As a result, we follow the authors of \cite{Ardila10} in order to define signed Minkowski sums.
	\end{enumerate}
\end{remark}

\begin{definition}\label{defn:H-v}
	For matrix $H \in \Z^{r \times n}$ whose rows are primitive vectors, and a vector $b\in \R^{r}$, let $P_{H,b}$ denote the possibly empty polytope given by
	\[ P_{H,b} = \{x \in \R^n \mid Hx \leq b\}. \]
	Suppose $H = \begin{bmatrix}
	h_{1} & h_{2} & \cdots & h_{r}
	\end{bmatrix}^{T}$, where $h_{i}$ is the $i$-th row vector of $H$. For any polytope $P$, let
	\[ v(H,P) = \begin{bmatrix}
		\max_{x \in P}{h_{1}\cdot x} & \max_{x \in P}{h_{2}\cdot x} & \cdots & \max_{x \in P}{h_{r}\cdot x}
		\end{bmatrix}^{T}. \]
	And let $b(H) = \{b\in \R^{r} \mid P_{H,b}\ne \emptyset \text{ and } v(H,P_{H,b}) = b \}$. 
\end{definition}

\begin{remark}\label{rem:H-b}
	A polytope $P$ could be obtained by different pairs of $(H,b)$. If an $H$ is given, $v(H,P)$ is the smallest vector $b$ such that $P=P_{H,b}$.
\end{remark}

\begin{lemma}\label{lem:id}
	Let $H \in \Z^{r\times n}$ be a matrix whose rows are primitive vectors, and $P\subseteq \R^{n}$ be a lattice polytope such that for every facet $F$ of $P$, $F$ is contained in a hyperplane $\{x\in \R^{n}\mid h\cdot x = c\}$ where $c\in \R$ and $h$ is a row vector of $H$. Then $P_{H,v(H,P)} = P$ and $v(H,P)\in b(H)$.
\end{lemma}

\begin{proof}
	For $1\le i\le r$ let $h_{i}$ be the $i$-th row of $H$. For any $x\in P$, we have $h_{i}\cdot x \le \max_{y\in P}{h_{i}\cdot y} = v(H,P)_{i}$, hence $H\cdot x\le v(H,P)$, by definition $x\in P_{H,v(H,P)}$. Therefore $P\subseteq P_{H,v(H,P)}$. Suppose there exists a point $z\in P_{H,v(H,P)} - P$. Then $z$ and $P$ are separated by at least one facet $F$ of $P$. So there exists $1\le i\le r$ and $c\in \R$ such that $h_{i}\cdot x\le c$ for all $x\in P$ but $h_{i}\cdot z>c$. Then $c\ge \max_{x\in P}{h_{i}\cdot x} = v(H,P)_{i}$. However, since $z\in P_{H,v(H,P)}$, we have $h_{i}\cdot z\le v(H,P)_{i}\le c$, a contradiction! So $P_{H,v(H,P)} = P$. By definition of $b(H)$, we have $v(H,P)\in b(H)$.
\end{proof}

The following lemma states that when $H$ is appropriately chosen, then Minkowski addition of polytopes is equivalent to vector addition. 

\begin{lemma}\label{lem:minkowski.sum}
Let $H$ be a matrix whose row vectors are all distinct primitive normal vectors of a polytope $P$ and $b = v(H,P)\in b(H)$. If $b_1,\dots,b_m \in b(H)$ such $b = \sum_{i=1}^{m}{b_{i}}$, then
\[P_{H,b} = \sum_{i=1}^{m}{P_{H,b_{i}}}. \]
In this case, for any constants $c_1, \dots, c_m > 0$, 
\[P_{H,\sum_{i=1}^{m}{c_{i}b_{i}}} = \sum_{i=1}^{m}{c_{i}P_{H,b_{i}}}. \]
\end{lemma}

\begin{proof}
	By \cite[Theorem 1.7.5]{Schneider93}, for any $x\in \sum_{i=1}^mP_{H,b_i}$ we have $Hx\le \sum_{i=1}^{m}{b_{i}}=b$, so $x\in P_{H,b}$. Hence $\sum_{i=1}^{m}{P_{H,b_{i}}} \subseteq P_{H,b}$. Conversely, suppose there exists a point $y\in P_{H,b} - \sum_{i=1}^{m}{P_{H,b_{i}}}$. Then $y$ and $\sum_{i=1}^{m}{P_{H,b_{i}}}$ are separated by some hyperplane. In other words, their exists a vector $l\in \R^{n}$ such that $l\cdot y > l\cdot x$ for all $x\in \sum_{i=1}^{m}{P_{H,b_{i}}}$. Now consider the following optimization problem: given $a\in \R^{r}$, for $x\in \R^{n}$, maximize $l\cdot x$ subject to $Hx\le a$. Since all constraints are linear, the solution could be obtain by Fourier-Motzkin elimination and the solution must be in the form $\sum_{j=1}^{n}{w_{j}a_{j}}$, where $w_{j}$'s are positive constants only depending on $H$ and $l$. Then we have
	
	\[l\cdot y \le \sum_{j=1}^{n}{w_{j}b_{j}}. \]
	
	Now for each $1\le i\le r$, choose $x_{i}\in P_{H,b_{i}}$ such that $l\cdot x_{i} = \sum_{j=1}^{n}{w_{j}(b_{i})_{j}}$. Then $x=\sum_{i=1}^{r}\in \sum_{i=1}^{m}{P_{H,b_{i}}}$ and
	\[l\cdot x = \sum_{i=1}^{r}{l\cdot x_{i}} =  \sum_{i=1}^{r}{\left(\sum_{j=1}^{n}{w_{j}(b_{i})_{j}}\right)} = \sum_{j=1}^{n}{w_{j}\left(\sum_{i=1}^{r}{(b_{i})_{j}}\right)} = \sum_{j=1}^{n}{w_{j}b_{j}}, \]
	a contradiction! So $\sum_{i=1}^{m}{P_{H,b_{i}}} = P_{H,b}$. The second claim follows from replacing $P$ with $P_{H,\sum_{i=1}^{m}{c_{i}b_{i}}}$ and $b_{i}$ with $c_{i}b_{i}$. 
\end{proof}

\begin{remark}\label{rem:tightness}
	Lemma \ref{lem:minkowski.sum} is no longer true if some $b_{i}$ does not belong to $b(H)$. Here is an example. Let $n=2, r=4$ and $P$ be the convex hull of $(0,0),(6,0),(1,5),(0,5)$, which is a right-angled trapezoid. Then
	
	\[ 
	H = \begin{bmatrix}
	-1 & 0 \\
	0 & -1 \\
	0 & 1 \\
	1 & 1
	\end{bmatrix}, b = (0,0,5,6)^{T}.\]
	
	Now let
	\[b_{1} = (0,0,2,4)^{T}, b_{2} = (0,0,3,2)^{T}. \]
	
	Then $b=b_{1}+b_{2}$ and $P_{H,b_{1}}$, $P_{H,b_{2}}$ are the convex hulls of $\{(0,0),(4,0),(2,2),(0,2)\}$ and $\{(0,0),(2,0),(0,2)\}$. Here note that $v(H,P_{H,b_{2}}) = (0,0,2,2)^{T}$, so $b_{2}\notin b(H)$. While 
	
	\[P_{H,b_{1}}+P_{H,b_{2}} = \conv{(0,0),(6,0),(2,4),(0,4)}\ne P. \]

\end{remark}

Proposition \ref{prop:minkowski.as.addition} below generalizes this result to signed Minkowski addition. This proposition serves two purposes. First, it is a vehicle to prove unique factorization. Second, it gives an algorithm to decompose a polytope into a signed Minkowski sum with respect to some given set of polytopes. 

\begin{proposition}\label{prop:minkowski.as.addition}
For $i = 1,\dots,m$, let $P_{H,b_i} \subset \Z^n$ be non-empty lattice polytopes where $H \in \Z^{r \times n}$ is the set of primitive normal vectors of their Minkowski sum $P := \sum_{i=1}^mP_{H,b_i}$. For $y^+,y^- \in \mathbb{N}^m$, suppose the signed Minkowski sum $\sum_{i=1}^{m}{(y^+_i-y^-_i)P_{H,b_i}}$ is well-defined. Then 
$$ \sum_{i=1}^m{(y^+_i-y^-_i)P_{H,b_i}} = P_{H,b}$$
where 
\begin{equation}\label{eqn:b.sum}
b = \sum_{i=1}^{m}{(y^+_i-y^-_i)b_i} \in \Z^r.
\end{equation}
\end{proposition}

\begin{proof}
Suppose $\sum_{i=1}^m(y^+_i-y^-_i)P_{H,b_i} = P$, where $P\subset \R^{n}$ is a nonempty polytope. By Definition \ref{defn:signed.minkowski}, 

\begin{equation}\label{eqn:minkowski.rewrite}
	\sum_{y^+_i>y^-_i}{(y^+_i-y^-_i)P_{H,b_i}} = \sum_{y^+_i<y^-_i}{(y^-_i-y^+_i)P_{H,b_i}} + P.
\end{equation}

Since $P\le P_{H, b^{+}_{i}}$, all primitive normal vectors of $P$ belong to $H$, hence there exists a vector $b\in b(H)$ such that $P=P_{H,b}$. By Lemma \ref{lem:minkowski.sum}, the LHS of (\ref{eqn:minkowski.rewrite}) is $P_{H, b^{+}_{i}}$, where $ b^{+}_{i}= \sum_{y^+_i>y^-_i}{(y^+_i-y^-_i)b_{i}}$. In addition, $\sum_{y^+_i<y^-_i}{(y^-_i-y^+_i)P_{H,b_i}} = P_{H, b^{-}_{i}}$, where $b^{-}_{i} = \sum_{y^+_i<y^-_i}{(y^-_i-y^+_i)b_{i}}$. Therefore the RHS of (\ref{eqn:minkowski.rewrite}) is $P(H, b^{-}_{i}+b)$ and $b^{+}_{i} = b^{-}_{i} + b$. So

\[b = b^{+}_{i} - b^{-}_{i} = \sum_{y^+_i>y^-_i}{(y^+_i-y^-_i)b_{i}} - \sum_{y^+_i<y^-_i}{(y^-_i-y^+_i)b_{i}} = \sum_{i=1}^{m}{(y^+_i-y^-_i)b_i}\in \Z^{r}. \]
\end{proof}

\section{Unit Polynomials and Bases}\label{sec:unit}
We now characterize and give conditions for unique factorizations for the set of units in $\NS,\ZS$ and $\ES$. Recall that $f$ is a unit if $\Delta_f$ is the trivial subdivision of its Newton polytope $\newt(f)$. By Corollary \ref{cor:factorize}, (rational) factorization of a unit is equivalent to (signed) Minkowski decomposition of its Newton polytope. Proposition \ref{prop:minkowski.as.addition} converts this problem to vector addition. In particular, unique factorization is possible if and only if the set of initial vectors forms a basis over $\mathbb{N}$ and $\mathbb{Z}$ for factorization and rational factorization, respectively. 

Throughout this section let $\S$ be a finite set of lattice polytopes in $\R^n$. Let $H(\S) \in \Z^{r \times n}$ be a matrix whose row vectors are all distinct primitive normal vectors of the polytope $\sum_{S \in \S}S$, with coordinate-wise lexicographic order. Then $H(\S)$ is uniquely defined.
 
Define
\begin{align*}
\B(\S) &= \{b \in \Z^{r}\cap b(H(\S)) \mid P_{H(\S),b} \in \S\} \label{eqn:b.s}, \\
\overline{\B}(\S) &= \{b \in \Z^{r}\cap b(H(\S))\mid P_{H(\S),b} \subset \Z^n \mbox{ is a non-empty lattice polytope}\}. \nonumber
\end{align*}

Recall that $\N\B$ denote the free module over $\mathbb{N}$, which is
\[\N\B(\S) = \left\lbrace \sum_{b \in \B(\S)}c_b \cdot b \in \Z^r: c_b \in \mathbb{N} \mbox{ for all } b \in \B(\S)\right\rbrace,\]
and $\Z\B(\S)$ is defined analogously.

\begin{remark}\label{rem:nonlattice}
	Note that even if $b\in \Z^{r}\cap b(H(\S))$, $P_{H(\S),b}$ could still not be a lattice polytope. Here is an example. Let
	\[\S = \{\conv\{(0,0), (1,2)\}, \conv\{(0,0), (-1,2)\} \}. \]
	Then
	\[\sum_{S\in \S}{S} = \conv\{(0,0),(1,2),(0,4),(-1,2)\} \]
	and
	\[H(\S) = \begin{bmatrix}
		-2 & -1 \\ -2 & 1 \\ 2 & -1 \\ 2& 1
	\end{bmatrix}. \]
	Now $b=(-1,2,0,3)^{T} \in \Z^{4}\cap b(H(\S))$, but
	\[P_{H(\S),b} = \conv\{(0.25,0.5),(0.75,1.5),(0.25,2.5),(-0.25,1.5)\}\]
	is not a lattice polytope.
\end{remark}

\begin{proposition}\label{prop:factor.unit}
Let $f$ be a unit tropical polynomial in $\R^n$, $P$ its Newton polytope. Then
\begin{enumerate}[(i)]
  \item $f \in \ES$ if and only if $P = P_{H(\S),b}$ for some $b \in \overline{\B}(\S)$
  \item $f \in \NS$ if and only if $P = P_{H(\S),b}$ for some $b \in \N\B(\S)$.
  \item $f \in \ZS$ if and only if $P = P_{H(\S),b}$ for some $b \in \overline{\B}(\S) \cap \Z\B(\S)$.
\end{enumerate}
\end{proposition}
\begin{proof}
	(i) By Corollary \ref{cor:factorize}, $f \in \ES$ if and only if there exist lattice polytopes $Q,R$ such that $P+Q=R$ and $R$ is a Minkowski sum of polytopes in $\S$. Suppose such $Q$ and $R$ exist. Let $b=v(H(\S),P)$. By \cite[Proposition 7.12]{ziegler1995lectures}, the normal fan of $R$ is a refinement of the normal fan of $P$. Hence $H(\S)$ and $P$ satisfies the condition of Lemma \ref{lem:id}, and thus $P=P_{H(\S),b}$ and $b\in b(H(\S))$. Hence $b \in \overline{\B}(\S)$.
	
	Conversely, suppose $P = P_{H(\S),b}$ for some $b \in \overline{\B}(\S)$. Let $S$ be the Minkowski sum of all polytopes in $\S$. Since $P = P_{H(\S),b}$, every facet of $P$ is parallel to some facet of $S$, hence the normal fan of $S$ is a refinement of the normal fan of $P$. By \cite[Page 318-319]{Grunbaum67}, there exists $\lambda>0$ such that $P$ is a Minkowski summand of $\lambda S$. Since any $\lambda'>\lambda$ works too, we may assume $\lambda \in \N$. Hence there exists another convex polytope $Q$ such that $P+Q=\lambda S$. Since both $\lambda S$ and $P$ are lattice polytopes, so is $Q$. In addition, $\lambda S$ is a Minkowski sum of polytopes in $\S$. Therefore we can take $R=\lambda S$.
	
	(ii) By Corollary \ref{cor:factorize}, $f \in \NS$ if and only if $P$ is a Minkowski sum of polytopes in $\S$. Suppose $P = \sum_{i=1}^{k}{P_{i}}$ where each $P_{i}\in \S$. Since $P_{i}\in \S$, all facets of $P_{i}$ belong to hyperplanes cut out by rows of $H(\S)$. We let $b_{i} = v(H(\S), P_{i})$, by Lemma \ref{lem:id}, $P_{i} = P_{H(\S),b_{i}}$ and $b_{i} \in b(H(\S))$, thus $b_{i}\in \B(\S)$. Now let $b=\sum_{i=1}^{k}{b_{i}}\in \N\B(\S)$. Note that $b = v(H(\S), P)$, by Lemma \ref{lem:id} again we have $P=P_{H(\S),b}$.
	
	Conversely, if $P=P(H(\S),b)$ for some $b\in \N\B(\S)$. Then there exists $b_{1},\cdot,b_{k}\in \B(\S)$ such that $b = \sum_{i=1}^{k}{b_{i}}$. Since $b_{i}\in b(H(\S))$ for $1\le i\le k$, by Lemma \ref{lem:minkowski.sum}, we have $P = \sum_{i=1}^{k}{P_{H(\S), b_{i}}}$ is a Minkowski sum of polytopes in $\S$.
	
	(iii) By Corollary \ref{cor:factorize}, $f \in \ZS$ if and only if there exist lattice polytopes $Q,R$ such that $P+Q=R$ and both $Q$ and $R$ are Minkowski sum of polytopes in $\S$. Suppose such $Q$ and $R$ exist. Then we can write
	
	\[Q = \sum_{i=1}^{k}{P_{H(\S),b_{i}}}, R = \sum_{j=1}^{l}{P_{H(\S),c_{j}}},\]
	
	where all $b_{i},c_{j}\in \B\S$. Now $R-Q=P$ is a well-defined polytope, by Proposition \ref{prop:minkowski.as.addition}, $P=P(H(\S), b)$, where $b = \sum_{j=1}^{l}{c_{j}} - \sum_{i=1}^{k}{b_{i}}$. So $b\in \Z\B(\S)$. In addition, since $P$ is a non-empty lattice polytope, $b\in \overline{\B}(\S)$.
	
	Conversely,	suppose $P = P_{H,b}$ for some $b \in \overline{\B}(\S) \cap \Z\B(\S)$. Since $b\in \Z\B(\S)$, there exists $k,l\in \N$ and vectors $b_{i},c_{j}\in \B\S$ for $1\le i\le k$ and $1\le j\le l$ such that $b = \sum_{j=1}^{l}{c_{j}} - \sum_{i=1}^{k}{b_{i}}$. Let $a=\sum_{j=1}^{l}{c_{j}} = b + \sum_{i=1}^{k}{b_{i}}$. By Lemma \ref{lem:minkowski.sum}, we have
	
	\[P_{H(\S),a} = \sum_{j=1}^{l}{P_{H(\S),c_{j}}} = P_{H(\S),b} + \sum_{i=1}^{k}{P_{H(\S),b_{i}}}. \]
	Since $b_{i},c_{j}\in \B\S$, by definition, $P_{H(\S),b_{i}}, P_{H(\S),c_{j}} \in \S$. Hence we can let $Q=\sum_{i=1}^{k}{P_{H(\S),b_{i}}}$ and $R=\sum_{j=1}^{l}{P_{H(\S),c_{j}}}$.
\end{proof}

Now we consider the problem of unique factorization. Note that both $\overline{\mathcal{B}}$ and $\mathcal{B}$ are integral vectors in $\R^r$ for some finite~$r$.  Thus, their $\Z$-modules $\mathbb{Z}\overline{\mathcal{B}}$ and $\Z\mathcal{B}$ are each isomorphic to some subgroup of $\mathbb{Z}^r$, so each must be finitely generated over $\Z$. 

\begin{definition}[Basis, Full Basis]\label{defn:basis}
Say that $\S$ is a basis if $\B(\S)$ is a basis over $\Z$ for $\Z\B(\S)$. Say that $\S$ is a full basis if $\B(\S)$ is a basis over $\Z$ for $\overline{\B}(\S)$. 
\end{definition}

\begin{proposition}[Unique Factorization]\label{prop:unique}
Let $\mathcal{S}$ be a set of polytopes. Then $\NS$ has unique factorization if and only if $\mathcal{S}$ is a basis.
\end{proposition}

\begin{proof}
Since $\NS$ contains the set of units polynomials, by Proposition \ref{prop:factor.unit}, $\S$ is a basis is a necessary condition for $\NS$ to have unique factorization. Now we prove sufficiency. Suppose $\mathcal{S}$ is a basis. Let $f \in \NS$. Suppose there are two factorizations of $f$. Multiply $f$ by a constant and a monomial if necessary, one can assume
$$f = f_1 \odot \dots \odot f_r = f'_1 \odot \dots \odot f'_{r'}$$  
for some units $f_i, f'_j$ with $\newt(f_i), \newt(f'_j) \in \mathcal{S}$, $i \in \{1,\ldots,r\}, j \in \{1,\ldots,r'\}$. By the Cayley trick, $\sum_{i=1}^r \newt(f_i)$ and $\sum_{j=1}^{r'} \newt(f'_j)$ are both equal to the support of $\Delta_f$. Since $\S$ is a basis, the sequence $(\newt(f_i))$ must equal the sequence $(\newt(f'_j))$, counting multiplicity. By the bijection given in the Cayley trick, the factorization of $f$ is uniquely determined by the sequence of polytopes to which $\Delta_f$ is a regular mixed subdivision. So $f$ has a unique factorization. 
\end{proof}

\begin{example}\label{ex:running.basis}
For $d = 2$, let $\S = \{S_1,\dots,S_{10}\}$ be the ten lattice polytopes shown in Figure \ref{fig:s}. Up to translation, this set $\S$ contains six primitive edges corresponding to the following vectors
\begin{equation}\label{eqn:s1.d2}
\S^1 = \{(0,1), (1,0), (1,1), (1,-2), (-2,1), (1,-1)\}. 
\end{equation}
The matrix $H(\S)$ has 12 row vectors, which are these six and their negatives
\begin{equation}\label{eqn:h.s}
(1,0), (0,1), (1,1), (1,-1), (1,2), (2,1). 
\end{equation}
The set $\overline{\mathcal{B}}(\S)$ consists of all lattice polytopes whose primitive eges are in $\S^1$. If $P \in \overline{\mathcal{B}}(\S)$ has $r$ edges, then one can list its consecutive edges, so that up to translation $P$ can be represented as a sequence of pairs $((w_1,s_1),\dots,(w_r,s_r))$, where $|w_i|$ is the length of the $i$-th edge of $P$ which is parallel to some $s_i \in \S^1$. Conversely, any such sequence with $\sum_{i=1}^rw_is_i = (0,0)$ and $\sum_{i \in I}w_is_i \neq (0,0)$ for all $I \subsetneq [r]$ defines a polytope in $\overline{\mathcal{B}}(\S)$ up to translation. By simple geometric arguments, one can derive the $H$-representation of $P$ from its edge sum sequence $((w_i,s_i))$, and thus prove that $P \in \overline{\mathcal{B}}(\S)$ if and only if $P \in \Z\S$. So $\S$ is a full basis. 
	
\begin{figure}[h]
\begin{center}
\includegraphics[width=0.75\textwidth]{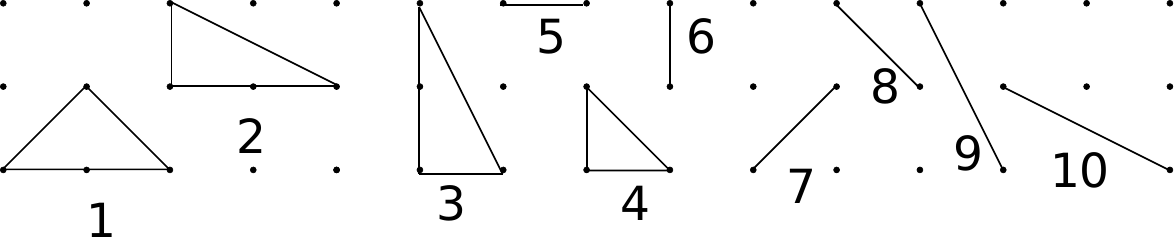} 
\end{center}
\vskip-0.5cm
\caption{A full basis $\S$.}\label{fig:s}
\end{figure}
\end{example}
\begin{example}\label{ex:sprime}
Figure \ref{fig:sprime} shows another full basis $\S'$ for the edge set $\S^1$ in \eqref{eqn:s1.d2}. Any polygon in $\Z^2$ with edges parallel to those in $\S^1$ has a unique decomposition in $\Z\S$ as well as $\Z\S'$. The decomposition with respect to one basis can be simpler. For example, let $P$ be the second polytope from the left of $\S'$ in Figure \ref{fig:s}. It has a trivial decomposition in $\Z\S'$, while its decomposition in $\Z\S$ is shown in Figure \ref{fig:s.prime}.
\begin{figure}[h]
\begin{center}
\includegraphics[width=0.75\textwidth]{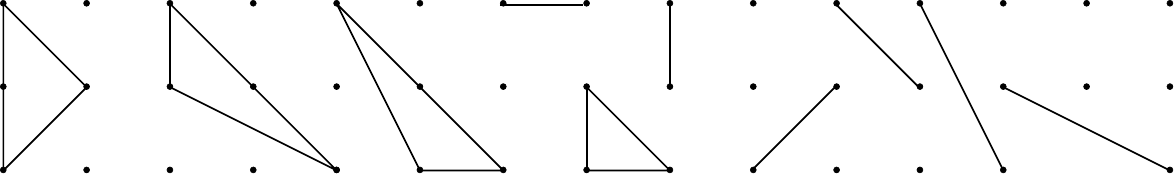} 
\end{center}
\caption{Another full basis $\S'$.}\label{fig:sprime}
\end{figure}
\begin{figure}[h]
\begin{center}
\includegraphics[width=0.75\textwidth]{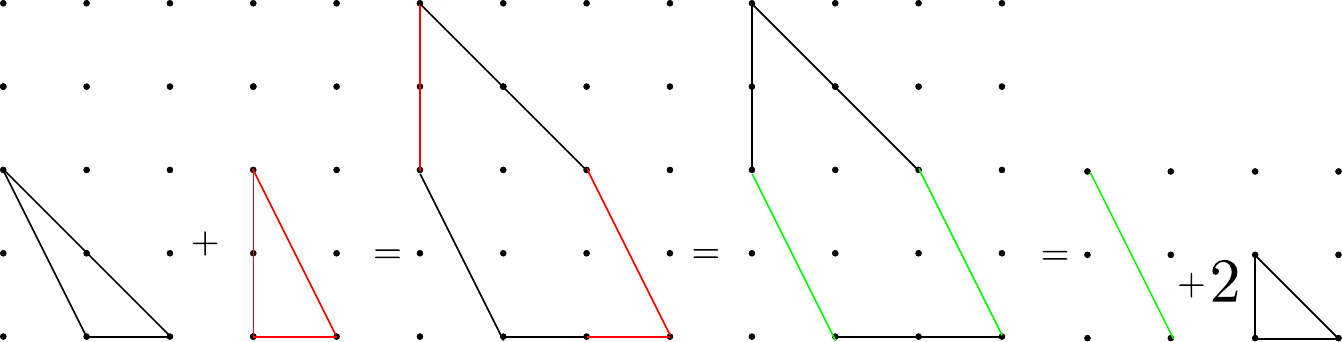} 
\end{center}
\caption{Decomposition of a polytope in $\Z\S$ with respect to two different full bases $\S$ and $\S'$.}\label{fig:s.prime}
\end{figure}

\end{example}

\section{Positive bases and local factorization}\label{sec:main}

Fix a finite set $\S$ of lattice polytopes in $\R^n$. In this section, we define the key concept of positive bases and prove Theorems \ref{thm:main} and~\ref{thm:SS}.

Since unique local factorization is a very strong criterion, the definition of positive bases is somewhat intricate. Throughout this section, we shall analyze the set of lattice polytopes $\S$ in Example \ref{ex:running.basis} as the running example.

\begin{definition}
Say that $\S$ is \emph{hierarchical} if $S \in \S$ implies all proper faces of $S$ are in $\mathbb{N}\S$. 
\end{definition}

\begin{example}
Let $\S$ be the ten lattice polytopes in Figure \ref{fig:s}. The edges of $S_1$ are integer multiples of $S_5,S_7$ and $S_8$. Similarly, edges of $S_2$ correspond to $S_5,S_6,S_{10}$, those of $S_3$ correspond to $S_5,S_6,S_9$, and those of $S_4$ correspond to $S_5,S_6,S_8$. So $\S$ is hierarchical. Note that $\S$ is still hierarchical if any polytopes from $S_1$ to $S_4$ were removed, although the remaining set of polytopes would no longer be a basis.
\end{example}

\begin{lemma}\label{lem:necessary}
Suppose $\NS$ has unique and local factorization. Then $\S$ is a hierarchical basis.
\end{lemma}
\begin{proof}
By Proposition \ref{prop:unique}, $\S$ must be a basis. For each $S \in \S$, let $f$ be a unit with $\newt(f) = S$. Then trivially $f \in \NS$. Since $\NS$ has local factorization, each lower-dimensional cell of $\Delta_f$ is in $\NS$, therefore proper faces of $S$ must be in $\mathbb{N}\S$. So $\S$ is hierarchical.
\end{proof}

Positive bases are hierarchical bases with an extra orientability condition. To define this, first we need the observation that the rows of the $H$-matrix of a hierarchical basis $\S$ come in pairs.

\begin{lemma}
Suppose $\S$ is hierarchical. If $v$ is a row vector of the $H$-matrix of $\S$, then $-v$ is also.
\end{lemma}
\begin{proof}
For each $S \in \S$, each edge $e$ of $S$ is a proper face of dimension 1. Thus $e \in \mathbb{N}\S$, which means $e$ is a positive integer multiple of some edge $w \in \S$ as vector. Let $\Sigma^1$ be the sum of all such shortest edges $w \in \S$. The $H$-matrix of $\S$ is that of the polytope $\sum_{S \in \S}S$, whose normal fan equals to the normal fan of $\Sigma^1$, which equals a hyperplane arrangement. So its $H$-matrix has the form stated in the lemma.
\end{proof}

\begin{definition}[Orientation]
Suppose $\S$ is hierarchical. Let $H$ be its $H$-matrix. An orientation $\tau$ is a map from row vectors of $H$ to $\{ \pm 1\}$, such that $\tau(v) = -\tau(-v)$. Given an orientation $\tau$, let $H^\tau_- = \{v: \tau(v) = -1\}$ and $H^\tau_+ = \{v: \tau(v) = 1\}$. Say that $\S$ is positive with orientation $\tau$ if for each $v \in H^\tau_-$ and $S \in \S$, $\face_v(S)$ is either a vertex of $S$, or is $S$.
\end{definition}

\begin{definition}[Positive basis]\label{def:positive.basis}
Say that $\S$ is a positive basis if it is a hierarchical basis and there exists some orientation $\tau$ such that it is  positive with respect to.
\end{definition}

\begin{example}
Let $\S$ be the full basis from Example \ref{ex:running.basis}, Figure \ref{fig:s}. Orient the vectors of $H(\S)$ as $-,-,+,-,+,+$ in the order that they are listed in \eqref{eqn:h.s}, and orient their negatives with opposite signs. The six edges of $\S$ are translations of the six edges listed in (\ref{eqn:s1.d2}). This shows that $\S$ is a positive basis. The full basis $\S'$ in Example \ref{ex:sprime} is positive with respect to orientation $\tau'$ which has signs $-,-,+,+,-,-$. Both $\S$ and $\S'$ are thus positive bases associated to the same $H$-matrix.
\end{example}

The positive orientation restricts when two polytopes in $\mathbb{N}\mathcal{S}$ can share a face. In particular, regular mixed subdivisions constructible from a sequence of polytopes in $\mathcal{S}$ must have a particular structure. The following gives an equivalent characterization in terms of pairs of polytopes in $\S$, without reference to an orientation $\tau$, and thus is easy to verify in specific examples. This result underpins Algorithm \ref{alg:positive.basis} for verifying whether a given set of lattice polytopes $\S$ is a positive basis.

\begin{proposition}[Characterization of positive basis]\label{prop:positive.basis}
Suppose $\S$ is a hierarchical basis with $H$-matrix $H$. Then $\S$ is a positive basis if and only if for each row vector $v$ of $H$, there are no polytopes $S,S' \in \S$, not necessarily distinct, such that both $\face_v(S)$ and $\face_{-v}(S')$ are proper faces of $S$ and $S'$ of dimension at least one, respectively. 
\end{proposition}
\begin{proof}
Suppose $\S$ is a positive basis with orientation $\tau$. For any pair of polytopes $S,S' \in \S$, either $\face_v(S)$ or $\face_{-v}(S)$ must be a proper face. So the criterion holds. Conversely, suppose the criterion holds. This means for each tuple $(v,-v,S)$, exactly one of the following cases hold
\begin{enumerate}
  \item[(i)] $\face_v(S)$ is a proper face and $\face_{-v}(S)$ is a vertex
  \item[(ii)] $\face_{-v}(S)$ is a proper face and $\face_v(S)$ is a vertex
  \item[(iii)] $\face_{-v}(S)$ and $\face_v(S)$ are both equal to $S$ or are both vertices of $S$.
\end{enumerate}
Define a local partial orientation $\nu: (v,S) \mapsto \{\pm 1, 0\}$ as follows. 
$$
\nu(v,S) = \left\{\begin{array}{cc}
+1 & \mbox{ if (i)} \\
-1 & \mbox{ if (ii)} \\
0 & \mbox{ else.}
\end{array}\right.
$$
Note that $\nu(-v,S) = -\nu(v,S)$. By the hypothesis, for each fixed $v$, there are no two polytopes $S,S' \in \S$ such that $\nu(v,S) = +1$ and $\nu(v,S') = -1$. Thus, one can define a global partial orientation $\tau': v \mapsto \{\pm 1, 0\}$ such that $\tau'(v) = +1$, $\tau'(-v) = -1$ whenever $v$ supports a proper face of some $S \in \S$, and $\tau'(v) = \tau'(-v) = 0$ if $v$ and $-v$ never support a proper face of $S$ for all $S \in \S$. Set $\tau: v \mapsto \{\pm 1\}$ by $\tau(v) = \tau'(v)$ if $\tau'(v) \neq 0$, otherwise choose $\tau(v) = +1, \tau(-v) = -1$ at random. Now take $v \in H^\tau_+$, and $S \in \S$. Only cases (i) and (iii) can happen. That is, $\face_{-v}(S)$ is not a proper face of $\S$. By definition, $\S$ is a positive basis with respect to orientation $\tau$. 
\end{proof}
\begin{corollary}\label{cor:subset.positive}
If $\S$ is a positive basis of $\Z\S$ and $\S' \subseteq \S$ is hierarchical, then $\S'$ is a positive basis of $\Z\S' \subseteq \Z\S$.
\end{corollary}

\begin{proof}[Proof of Theorem \ref{thm:main}]
Let $\S$ be a positive basis, and $f$ be a tropical polynomial such that each cell of $\Delta_f$ is a Minkowski sum of some polytopes in $\S$. As $\mathcal{S}$ is a basis, by Proposition~\ref{prop:factor.unit}, this decomposition is unique. Let $\S(f)$ denote the sequence of polytopes in $\S$ that appear as Minkowski summands of the maximal cells of $\Delta_f$, with multiplicity. Note that $|\S(f)|$ is finite. We shall do induction on $|\S(f)|$. If $|\S(f)| = 1$, then $\Delta_f$ is the trivial subdivision of a single polytope in $\S$, so we are done. If $|\S(f)| > 1$, pick $F \in \S(f)$ of maximal dimension. We shall use Corollary \ref{cor:factorize} to show that $\T(f) = \T(F) \cup \T(f')$ for some polynomial $f'$. We then argue that each cell of $\Delta_{f'}$ is still in $\mathbb{N}\mathcal{S}$, and $\S(f') \subset \S(f)$, so $|\S(f')| < |\S(f)|$. This would complete the induction step. 
Let us prove the first claim that the condition of Corollary \ref{cor:factorize}| holds for $\Delta_f$. By the setup, there exists a cell~$\sigma \in \Delta_f$ with $F\leq \sigma$. Since $\sigma \in \mathbb{N}\S$, maximal faces of $\sigma$ are supported by vectors in $H(\S)$. Let $v \in H(\S)$ be such a vector. We need to show that all cells in direction $v$ from $\sigma$ has $\face_v(F)$ as a Minkowski summand (recall Definition \ref{defn:neighbor}). By induction on the number of neighbors, it suffices to show that the immediate neighbor of $\sigma$ in direction $v$ has this property. If $\sigma$ does not have a neighbor in direction $v$, then we are done. Otherwise, let $\sigma'$ be this neighbor, that is, $\face_v(\sigma) = \face_{-v}(\sigma')$. 
If $\face_v(F)$ is a point, then this is trivial. If $F \leq \sigma'$, then trivially $\face_v(F) \leq \face_v(\sigma')$. Therefore, we are left with the case that $\face_v(F)$ is not a vertex, and $F \not\leq \sigma'$. Write
$$ \sigma' = \sum_{S \in V^\perp} y_S S + \sum_{T \in V} y_T T $$
for unique $y_S, y_T \in \mathbb{N}$, where 
$$V^\perp = \{S \in \S(f): \face_v(S) \mbox{ is a proper face of } S, S \neq F\} $$
and
$$ V = \{T \in \S(f): \face_v(T) = T = \face_{-v}(T), T \neq F\}.$$ 
By Proposition \ref{prop:minkowski.as.addition}, $F \leq \sigma$ implies
$$ \face_v(F) \leq \face_v(\sigma) = \face_{-v}(\sigma') = \sum_{S \in V^\perp} y_S \face_{-v}(S) + \sum_{T \in V} y_T T. $$
We now argue that $\face_v(F) \leq \sum_{T \in V} y_T T$. If this holds, then $\face_v(\sigma') \geq \sum_T y_T T$ so $\face_v(F) \leq \face_v(\sigma')$ as needed. 
Suppose for contradiction that this does not hold. There is at least one $S \in V^\perp$ such that $\face_{-v}(S)$ is a proper face of $\S$, and $c \cdot \face_v(F) \geq \face_{-v}(S)$ for some $c \in \mathbb{N}$. If $\face_v(F)$ is also a proper face of $F$, then $\S$ cannot be a positive basis by Proposition~\ref{prop:positive.basis}. So we must have $\face_v(F) = F$. 
As $\S$ is a hierarchical basis, $F \in \S$ and $\face_{-v}(S) \in \mathbb{N}\S$, $c \cdot F \geq \face_{-v}(S)$ implies $\face_{-v}(S) = F$. But this means $S$ has dimension strictly larger than that of $F$, which is not possible as $F$ is maximal amongst those in $\S(f)$.
 So we obtain the desired contradiction. This proves the first claim.   
For the second claim on $|\S(f')|$, note that cells of $\Delta_{f'}$ are either equivalent to some cells of $\Delta_f$, or they have the form
$\tau' \equiv \tau - \omega$ for some $\tau \in \Delta_f$ and some face $\omega$ of $F$. Since $\S$ is hierarchical, $\tau, \omega \in \mathbb{N}\S$. Since $\S$ is a basis, $\tau' \in \mathbb{N}\mathcal{S}$. So all cells of $\Delta_f$ are in $\mathbb{N}\S$, and $\S(f') \subseteq \S(f)$. But $F \in \S(f)$ and $F \notin \S(f')$, so $\S(f') \subset \S(f)$, and thus $|\S(f')| < |\S(f)|$. This concludes the proof.
\end{proof}

\begin{proof}[Proof of Theorem \ref{thm:SS}]
Suppose $f \in \ZS$. Let $g \in \NS$ be such that $h = g \odot f \in \NS$. For each cell $\sigma_f$ of $\Delta_f$, there exists cells $\sigma_g$ of $\Delta_g$ and $\sigma_h$ of $\Delta_h$ such that
$$ \sigma_f + \sigma_g = \sigma_h. $$
By the Cayley trick, $\sigma_g, \sigma_h \in \mathbb{N}\S$, so $\sigma_f \in \Z\S$. Conversely, suppose all cells of $\Delta_f$ are in $\Z\S$. Compute the signed Minkowski sum decomposition of each cell of $\Delta_f$ with respect to $\mathcal{S}$. Let $\S^-(f)$ be the sequence of polytopes in $\S$ that appear with negative signs, with multiplicity. Similar to the proof of Theorem \ref{thm:main}, we shall do an induction on $|\S^-(f)|$. If $\S^-(f) = \emptyset$ then all cells of $\Delta_f$ are in $\mathbb{N}\S$. By Theorem~\ref{thm:main}, $f \in \NS$, so we are done. If not, for $S$ a polytope of maximal dimension in $\S^-(f)$, let $\sigma_S$ be the cell of $\Delta_f$ where $-y_SS$ appears in its signed Minkowsi decomposition for some $y_S > 0$. Define a unit $g(S)$ such that $\newt(g(S)) = S$, and that $g^\ast$ is a classical linear function such that $g^\ast = f^\ast$ restricted to $\sigma_S$. Let $f' := f \odot (g(S))^{\odot y_S}$. Since $\S$ is hierarchical, $S \in \mathbb{N}\S$ implies that its faces are in $\mathbb{N}\S$. So cells of $\Delta_{f'}$ are in $\Z\S$, and $\S^-(f') \subsetneq \S^-(f)$. As $\S^-(f)$ is a finite sequence, by induction we are done. So $f \in \ZS$, which proves that $\ZS$ has local factorization. Note that our proof produces a polynomial $g \in \NS$ such that $f \odot g \in \NS$. For uniqueness of this $g$, it is sufficient to show that this $g$ does not depend on the order amongst polytopes of maximal dimension in $\S^-(f)$. Indeed, note that if $S,S' \in \mathbb{N}\S$ are two polytopes of the same dimension, and $\omega$ is a proper face of $S$, then $\omega \neq S'$. Therefore, if $S, S'$ are two maximal dimensional polytopes in $\S^-(f)$, $S \neq S'$, then $S' \in \S^-(f \odot g(S))$. So the $g$ produced by the proof is unique. Furthermore, any other $\tilde{g} \in \NS$ such that $f \odot \tilde{g} \in \NS$ must contain enough units to bring all cells of $\Delta_f$ from $\Z\S\backslash \mathbb{N}\S$ to $\mathbb{N}\S$, and therefore must contain $g$ in its factorization. So the $g$ produced is the minimal denominator. Finally, let us prove the assertion on full positive basis. Suppose $\ZS = \ES$. In particular, $\Z\S = \mathbb{E}\S$, so $\S$ is a full basis by Definition \ref{defn:basis}. Conversely, suppose $\S$ is a full positive basis. For $f \in \ES$, let $g$ be a product of units such that $f \odot g \in \NS$. Then edges in $\Delta_f$ must be parallel to integer multiples of primitive edges in $\Delta_{f \odot g}$, which are contained in $\S^1$. Therefore, each cell of $\Delta_f$ is in $\mathbb{E}\S$. But $\S$ is a full basis, so 
each cell of $\Delta_f$ is also in $\Z\S$. As $\S$ is a positive basis, $\ZS$ has local factorization so $f \in \ZS$. 
\end{proof}

\begin{example}[Factorization into tropical plane curves of degree 2]\label{ex:d2}
Let $\S$ be the positive basis of Example \ref{ex:running.basis}. For concreteness, we fix an ordering on the rows of the $H$-matrix, so that it is the transpose of the following $3\times 14$ matrix:
	
	\begin{equation}\label{eqn:H.matrix}
		\begin{bmatrix}
		1 & -1 & 0 & 0  & 1 & -1 & 1  & -1 & 1 & -1 & 2 & -2 & 1 & -1 \\
		0 & 0  & 1 & -1 & 1 & -1 & -1 & 1  & 2 & -2 & 1 & -1 & 1 & -1 \\
		0 & 0  & 0 & 0  & 0 & 0  & 0  & 0  & 0 & 0  & 0 & 0  & 1 & -1 
		\end{bmatrix}.
	\end{equation}
	Let $f_{q}(x_1,x_2,x_3)=\max(2x_1+2x_2, x_1+3x_2-2, x_1+x_2+2x_3-3, 3x_1+x_3-1, x_1+2x_2+x_3-4, 4x_1-3)$. Its regular subdivision $\Delta_{f_q}$ consist of $3$ maximal cells $C_{1},C_{2},C_{3}$ shown in Figure \ref{fig:h_q}. Note that these cells are in $\Z\S$, with signed Minkowski decomposition
	\begin{eqnarray*}
		C_{1}&=&P_1-P_2+P_3-P_4+P_{10}+(1,0,1),  \\
		C_{2}&=&-P_1+2P_4+P_7+(1,1,-2),  \\
		C_{3}&=&-P_3+2P_4+P_9+(2,0,-2).  \\
	\end{eqnarray*}
By Theorem \ref{thm:SS}, $f_q$ admits a rational factorization $h_q = f_q \odot g_q$, where $g_q,h_q \in \NS$, that is, they are products of tropical quadratics in three variables. Indeed, Algorithm \ref{alg:gen-ZS} outputs
	\[
	\begin{split}
	&g_q(x) = \max(2x_1, 2x_3-10/3, x_2+x_3-2) \\ 
&+ \max(x_1+x_3, 2x_3-5/3, x_2+x_3-1/3) \\
&+ \max(2x_3, 2x_2-1, x_1+x_3-2) +\max(2x_1, 2x_3-5, x_1+x_2-2).
	\end{split}
	\]
	The product $h_q = f_q \odot g_q$ has degree 12. Algorithm \ref{alg:gen-NS} shows it is factorizable as product of seven $\S$-units. The decompositions of $\newt(h_q)$ are shown in Figure \ref{fig:h_q}. One can readily check from the figure that $g_q$ is the minimal polynomial in $\NS$ such that $f_q \odot g_q \in \NS$.
	\begin{equation*}
		\begin{split}
		&h_{q}(x)  = x_{1}-3x_{3} + 2\max(2x_3-\frac{5}{2},x_1+x_3,x_2+x_3-2) \\
&+\max(2x_3-\frac{8}{3},x_1+x_3-1,2x_2) + \max(x_2+x_3-2,2x_1)\\
&+2\max(2x_3,x_1+x_3-2,x_2+x_3-\frac{1}{2}) + \max(2x_3-\frac{10}{3},x_1+x_2-\frac{1}{3},2x_1). 
		\end{split}
	\end{equation*}
\end{example}
\begin{figure}[ht]

\centering	
\begin{minipage}[t]{0.44\textwidth}
	\includegraphics[width=1\textwidth]{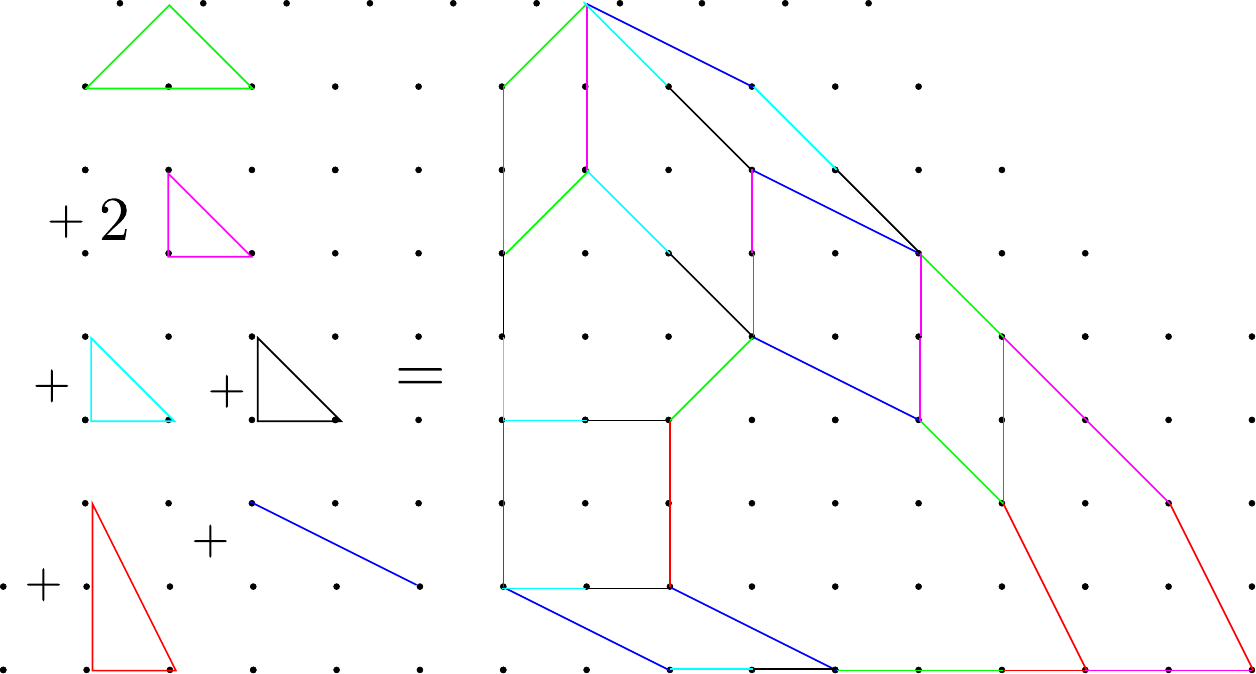}
\end{minipage}
\hspace{0.2em}
\begin{minipage}[t]{0.5\textwidth}
	\includegraphics[width=1\textwidth]{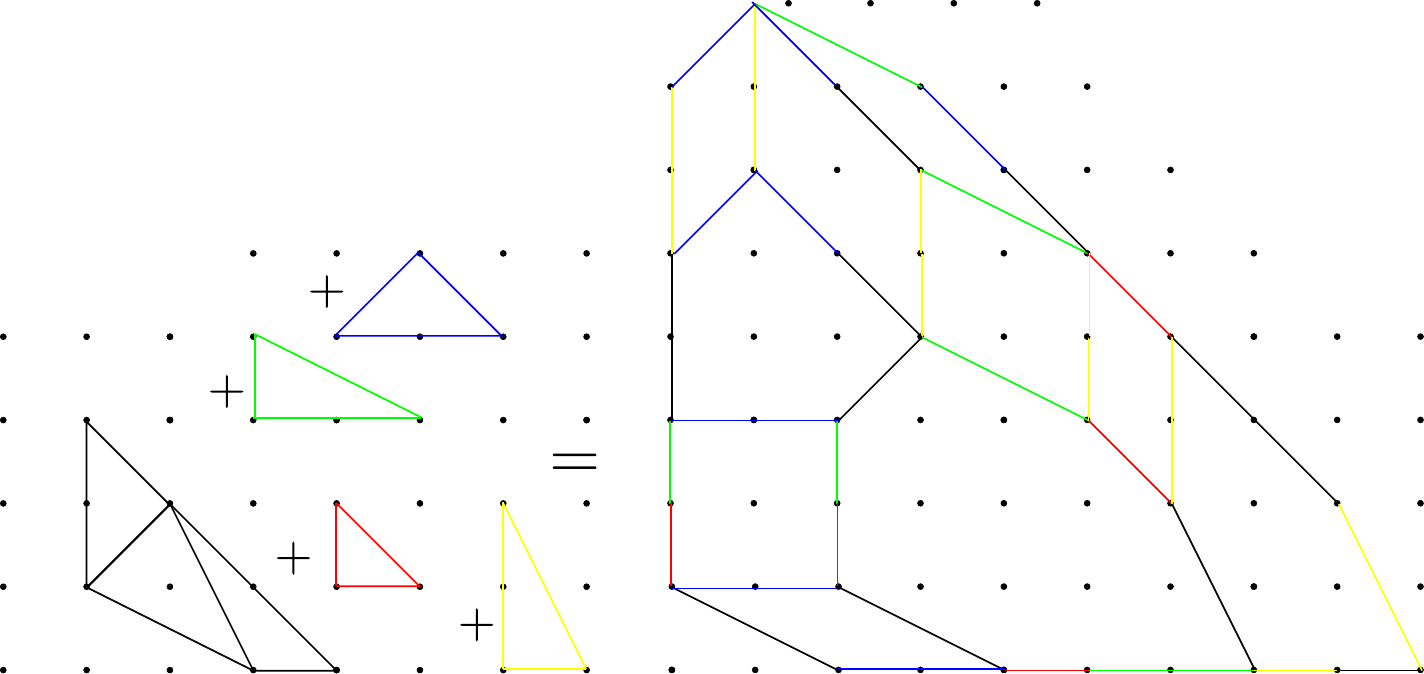}
\end{minipage}

\caption{An illustration of Theorem \ref{thm:SS}. Two ways to decompose $\Delta_{h_q}$ in Example \ref{ex:d2}: by writing $h_q$ as a product of units, or by writing $h_q = f_q \odot g_q$. In particular, this shows that $f_q$ is rationally factorizable.}\label{fig:h_q}
\end{figure}

\section{Two families of full positive bases}\label{sec:positive.basis}

\begin{lemma}\label{lem:simplex}
Let $\mathcal{S} = \{\Delta_I: I \subseteq [n], |I| \geq 2\}$ be the set of the standard simplex in $\R^n$ and its proper faces. Then $\mathcal{S}$ is a full positive basis.
\end{lemma}
\begin{proof}
The $H$-matrix of $\S$ consists of vectors of the form $v_I = \sum_{i\in I}e_i$ where $I \subseteq [n]$ and their negatives. It is straight forward to verify that $\S$ satisfies the criterion of Proposition \ref{prop:positive.basis}, so $\S$ is a positive basis. By \cite[Proposition 2.4]{Ardila10}, $P$ has edges parallel to $e_i - e_j$ if and only if $P \in \Z\S$. Thus $\S$ is full. \end{proof}

\begin{proposition}[Graphical bases]\label{prop:graphical}
Given a graph $\mathbb{G}$ on $n$ nodes, let $\S_{\mathbb{G}}$ consist of simplices $\Delta_I$, where $I \subseteq [n]$ runs over all cliques in $\mathbb{G}$. Then $\S_\mathbb{G}$ is a full positive basis.
\end{proposition}
\begin{proof}[Proof of Proposition \ref{prop:graphical}]
By Lemma \ref{lem:simplex}, $\S_{K_n}$ is a full positive basis. Now, $\mathbb{G}$ is a subgraph of $K_n$, so $\S_\mathbb{G}\subseteq \S_{K_n}$.
Clearly $\S_\mathbb{G}$ is hierarchical, so by Corollary~\ref{cor:subset.positive}, $\S_\mathbb{G}$ is a positive basis. It remains to show that it is full. It is sufficient to show that if $P$ is a lattice polytope whose edges are parallel to $e_i - e_j$ for $(i,j) \in e(\mathbb{G})$, then $P \in \Z\S_\mathbb{G}$. Suppose $P$ is a lattice polytope with such edge directions. By Lemma~\ref{lem:simplex}, $P \in \Z\S(K_n)$, so
$$ P = \sum_{I \subseteq [n]} y_I(P)\Delta_I \in \Z\S(K_n) $$
for unique constants $y_I(P) \in \Z$, $I \subseteq [n]$. Let 
$$\mathcal{I} = \{I \subseteq [n]: \Delta_I \notin \S_\mathbb{G}, y_I(P) \neq 0\}.$$
If $\mathcal{I} = \emptyset$, then we are done. Otherwise, for contradiction, consider two cases. 
\begin{itemize}
  \item There exists some $I \in \mathcal{I}$ such that $y_I(P) > 0$. Then there is some edge $(i,j) \notin e(\mathbb{G})$ such that $i,j \in I$. But $P$ must contain an edge parallel to $e_i - e_j$, a contradiction. 
  \item For all $I \in \mathcal{I}$, $y_I(P) < 0$. Let 
$$P' := P + \sum_{I \in \mathcal{I}} (-y_I(P))\Delta_I.$$ 
Then $P' = \sum_{L \in \S_\mathbb{G}}y_L(P)\Delta_L$, so edges in $P'$ are parallel to $e_i - e_j$ for $(i,j) \in e(\mathbb{G})$. On the other hand, since $\mathcal{I} \neq \emptyset$, there exists some $I \in \mathcal{I}$ such that $\Delta_I \leq P'$, so $P'$ must contain an edge parallel to $e_i - e_j$ for some $(i,j) \notin e(\mathbb{G})$, a contradiction. 
\end{itemize}
Therefore, one must have $P \in \Z\S_\mathbb{G}$. So $\S_\mathbb{G}$ is a full basis.
\end{proof}

Next we show that there are full positive bases in $\Z^2$ starting from any given set of primitive edges $\S^1$. As there are many full positive bases for a given set $\S^1$, we deliberately present a non-constructive proof. In specific examples, it is not difficult to construct a given full positive basis in $\Z^2$, see Example \ref{ex:d2}.

\begin{proposition}\label{prop:n2}
Let $\S^1$ be a set of primitive edges in $\Z^2$. There exists a full positive basis $\S$ of $\mathbb{E} \S^1$. 
\end{proposition}
\begin{proof}
If $\S^1$ has cardinality one or two, take $\S = \S^1$ and the result holds trivially. Now suppose $\S^1$ consists of at least three edges. The row vectors of $H(\S^1)$ consists of primitive vectors in $\Z^2$ which are orthogonal to those in $\S^1$. Since $\S^1$ has at least three edges, one can choose an orientation $\tau$ such that $\spanv_\R(H^\tau_+) = \R^2$. Let $\mathcal{P}^\tau$ be the set of all full-dimensional lattice polygons whose outer normal vectors are nonnegative integer multiples of those in $H^\tau_+$. Choose $\S' \subseteq \mathcal{P}^\tau$ so that $\S'$ is a basis for $\Z\mathcal{P}^\tau$. Set $\S = \S' \cup \S^1$. Then $\S$ is a hierarchical basis. We claim that $\S$ in fact generates $\mathbb{E}\S^1$. Indeed, let $P$ be a polygon in $\mathbb{E}\S^1$. Let $v_-(P)$ be the set of outer normal vectors of $P$ that are positive integer multiples of vectors in $H^\tau_-$. If $v_-(P) = \emptyset$, then $P \in \mathcal{P}^\tau$, so $P \in \Z\S$. Otherwise, for each~$c_v \cdot (-v) \in v_-(P)$ with $-v \in H^\tau_-$, there is a polygon $Q(v) \in \mathcal{P}^\tau$ such that $c_v \cdot v$ is an outer normal vector of $Q(v)$. Then
\begin{equation}\label{eqn:pq}
P + \sum_{v \in v_-(P)}Q(v) = P' + \sum_{v \in v_-(P)} c_v \cdot e(v),
\end{equation}
where $e(v) \in \S^1$ is the edge orthogonal to $v$, and $P'$ is some polytope whose outer normal vectors are all in $H^\tau_+$. Thus the RHS of (\ref{eqn:pq}) is in $\Z\S$, and each $Q(v)$ is in $\Z\S$, so~$P \in \Z\S$. Thus $\S$ is a full positive basis of $\mathbb{E}\S^1$. 
\end{proof}

\begin{corollary}[Rational Factorization for Bivariates]\label{cor:main.rational}
Any bivariate tropical polynomial is rationally factorizable into a product of affine monomials of the form $(x,y) \mapsto c_0\oplus x^{\odot a}\odot y^{\odot b}$, $a,b \in \mathbb{N}, c_0 \in \R$.
\end{corollary} 
\begin{proof}[Proof of Corollary \ref{cor:main.rational}]
Let $e(f)$ be the set of primitive edges in $\Delta_f$, $\Sigma$ be their Minkowski sum. By Proposition \ref{prop:n2}, there exists a full positive basis $\S$ such that~$\S^1 = e(f)$, so $f \in \ZS = \ES$. Note that by definition, $\mathbb{E}[\S^1] = \mathbb{E}[\S]$. Thus $f \in \mathbb{E}[\S^1]$, that is, $f$ is rationally factorizable into a product of affine monomials whose Newton polygon are integer multiples of the edges in $\S^1$. 
\end{proof}

\section{Algorithms}\label{sec:algorithm}

In this section we discuss various algorithms for factorization, rational factorization and their implementations. Without loss generality, we assume that the tropical polynomials of interest are homogeneous. Given such a tropical polynomial $f$ and a finite set of lattice polytopes $\S$, we supply algorithms to do the following

\begin{enumerate}
  \item Decide whether $\S$ is a positive basis (Algorithm \ref{alg:positive.basis}).
  \item Given $\S$ a positive basis, decide whether $f\in \ZS \backslash \NS$, $f \in \NS$, or neither (Algorithm \ref{alg:gen-test}).
  \item If $f \in \NS$, produce the unique factorization for $f$ (Algorithm \ref{alg:gen-NS}).
  \item If $f \in \ZS \backslash \NS$, produce a $g \in \NS$ such that $f\odot g\in \NS$ (Algorithm \ref{alg:gen-ZS}).
\end{enumerate} 

There are a few subroutines of polytopes used in our algorithms. All these subroutines can be done using the software {\tt Polymake}\cite{polymake}. They include (all input $P$ are in V-representation)

\begin{itemize}
	\item Input a polytope $P$, return its H-representation. In particular, its H-matrix $H(P)$.
	\item Input a polytope $P$, return the set of its faces $Faces(P)$, each in V-representation.
	\item Input a polynomial $f$, output its Newton polytope $\Newt(f)$.
	\item Input a tropical polynomial $f$ with a weight at each vertex of $P$, return the maximal polytopes in its regular subdivision with respect to the weights, denoted $\Delta_{f}$.
	\item Input a polytope $P$ and a nonzero vector $v$, return $\face_{v}(P)$.
\end{itemize}

Algorithm \ref{alg:positive.basis} is based on Proposition \ref{prop:positive.basis}. As a by-product, it also computes the $H$-matrix of $\S$, which is needed for subsequent algorithms. Computing $H(\S)$ is in fact the most intensive part, as this matrix could have exponentially many row vectors. The other algorithms are based on Proposition \ref{prop:factor.unit} and the proofs of Theorems \ref{thm:main} and \ref{thm:SS}. A crucial difference to the proof is that these algorithms use the stricter notion of $=_2$ instead of $=_3$ for equality of two tropical polynomials. This means we need to keep track of translations of the regular subdivisions and translations of the Legendre transform. A major part of the algorithms is spelling out the details of this step.

\begin{algorithm}[H]
	\begin{algorithmic}[1]
	\caption{\texttt{Is Positive Basis}} \label{alg:positive.basis}
	\Function{IsPositiveBasis}{$\S$} \\
	\textbf{Input}: $\S$, a finite set of lattice polytopes, each in V-representation \\
	\textbf{Output}: True if $\S$ is a positive basis; False otherwise
	\For{$F$ in $Faces(\S)$}
		\If{not $F \in \S$} 
		\EndIf
	\EndFor
	\State $S \leftarrow \sum_{P \in \S} P$
	\State $H \leftarrow H(S)$
	\If{$rank(H) \ne$ number of rows in $H$} \textbf{return} False
	\EndIf
	\For{row vector $v$ in $H$}
		\State $f(v) \leftarrow \emptyset$
		\For{$P$ in $\S$}
			\If{$\face_{v}(P) \in Faces(P)$ and $\face_{v}(P)\ne P$}
				\State $f(v) \leftarrow f(v) \cup \{P\}$
			\EndIf
		\EndFor
	\EndFor
	\For{row vector $v$ in $H$}
		\If{$f(v)\cap f(-v) \ne \emptyset$} \textbf{return} False
		\EndIf
	\EndFor
	\State \textbf{return} True
	\EndFunction
	\end{algorithmic}
\end{algorithm}

\begin{algorithm}[H]
	\begin{algorithmic}[1]
		\caption{\texttt{Membership test for $\ZS$ and $\NS$}} \label{alg:gen-test}
		\Function{Membership}{$\S, f$} \\
		\textbf{Input}: A tropical polynomial $f$ and a positive basis $\S$ \\
		\textbf{Output}: $f\notin \ZS$; $f\in \ZS/\NS$; $f\in \NS$ with decomposition of all polytopes in $\Delta_{f}$ as Minkowski sums of polytopes in $\S$
		\State $H \leftarrow H(\S)$
		\State $status \leftarrow 1$
		\For{$\sigma \in \Delta_{f}$}
			\State $b(\sigma)\leftarrow v(H, \sigma)$		
			\If{$b(\sigma)\notin \Z\B(\S)$} \textbf{return} $f\notin \ZS$
			\ElsIf{$b(\sigma)\notin \N\B(\S)$}
				\State $status \leftarrow 0$
			\Else
				\State Write $b(\sigma) = \sum_{P\in \S}{\sigma_{P}\cdot P}$, where each $\sigma_{P}\in \N$
			\EndIf
		\EndFor
		\If{$status = 0$}
			\State \textbf{return} $f\in \ZS/\NS$
		\Else
			\State \textbf{return} $\sigma_{P}$ for all $\sigma \in \Delta_{f}$ and $P\in \S$
		\EndIf
		\EndFunction
	\end{algorithmic}
\end{algorithm}

\begin{algorithm}[H]
	\begin{algorithmic}[1]
		\caption{\texttt{Factorization for $\NS$}} \label{alg:gen-NS}
		\Function{FactorN}{$\S, h$} \\
		\textbf{Input}: a positive basis $\S$ and a tropical polynomial $h\in \NS$ \\
		\textbf{Output}: $\S$-units $h_{i}$ with multiplicities $m_{i}$ such that $h=_{2} \odot_{i}{h_{i}^{\odot m_{i}}}$.
		\State $H \leftarrow H(\S)$
		\State $(\sigma_{P})_{\sigma \in \Delta_{h}, P\in \S} \leftarrow Membership(\S, h)$
		\State $s \leftarrow \emptyset$
		\For{vertex $V\in \sigma$}
			\State $s \leftarrow s\cup \{\sum_{i=1}^{n}{V_{i}\cdot \sigma_{i}} = h(V)\}$
		\EndFor
		\State Solve $s$ for the unique solution $(\sigma_{1},\ldots,\sigma_{n})$ \Comment{uniqueness guaranteed by the fact that $\sigma$ has maximal dimension)}
		\State $l_{\sigma} \leftarrow$ linear function $\R^{n}\to \R$ with $l_{\sigma}(x_{1},\ldots,x_{n}) = \sum_{i=1}^{n}{\sigma_{i}\cdot x_{i}}$
		\State $\mathcal{O}\leftarrow \emptyset$ \Comment{output of factors with multiplicity}
		\For{$P\in \S$}
			\For{$\sigma \in \Delta_{h}$}
				\State $a_{\sigma, P} \leftarrow \sigma_{P}$
			\EndFor
		\EndFor
		\While{$\exists a_{\sigma, P}>0$}
			\State $S\leftarrow$ an element of $\{P\in \S \mid \exists \sigma\in \Delta_{h} \text{ such that } a_{\sigma, P}>0 \}$ with maximal dimension
			\State $\eta \leftarrow$ a polytope such that $a_{\eta, S}>0$
			\State Add the following $\S$-unit to $\mathcal{O}$:
				\[x\mapsto \max_{\text{vectex } v \in \eta}{\left(v\cdot x + l_{\eta}(v)\right)} \]
			\For{$\sigma \in \Delta_{h}$}
				\State $J_{S, \sigma} \leftarrow \conv\left(\argmax_{\text{vectex } v \in \eta}{\left(l_{\eta}(v) - l_{\sigma}(v)\right)}\right)$
				\State $b_{S, \sigma} \leftarrow v(H, J_{S, \sigma})$ 
				\State Write $b_{S, \sigma} = \sum_{P\in \S}{(J_{S, \sigma})_{P}\cdot P}$ \Comment{$b_{S, \sigma} \in \N\B(\S)$}
				\For{$P\in \S$}
					\State $a_{\sigma, P} \leftarrow a_{\sigma, P} - (J_{S, \sigma})_{P}$
				\EndFor
			\EndFor
		\EndWhile
		\State \textbf{Return} $\mathcal{O}$
		\EndFunction
	\end{algorithmic}
\end{algorithm}

\begin{algorithm}
	\begin{algorithmic}[1]
		\caption{\texttt{Factorization for $\ZS$}} \label{alg:gen-ZS}
		\Function{FactorZ}{$\S, f$} \\
		\textbf{Input}: a positive basis $\S$ and a tropical polynomial $f\in \ZS$ \\
		\textbf{Output}: a tropical polynomial $g\in \NS$ such that $f\odot g\in \NS$.
		\State $H \leftarrow H(\S)$
		\State $m \leftarrow Membership(\S,f)$
		\If{$m = f\in \ZS/\NS$}
			\For{$\sigma \in \Delta_{f}$}
				\State $b(\sigma)\leftarrow v(H, \sigma)$
				\State Write $b(\sigma) = \sum_{P\in \S}{\sigma_{P}\cdot P}$, where each $\sigma_{P}\in \Z$
			\EndFor
			\State $L \leftarrow \emptyset$
			\For{$\sigma \in \Delta_{f}$}
				\For{$P\in \S$}
					\If{$\sigma_{P}<0$}
						\State $s \leftarrow \emptyset$
						\For{vertex $V\in \sigma$}
							\State $s \leftarrow s\cup \{\sum_{i=1}^{n}{V_{i}\cdot \sigma_{i}} = f(V)\}$
						\EndFor
						\State Solve $s$ for the unique solution $(\sigma_{1},\ldots,\sigma_{n})$ \Comment{uniqueness guaranteed by the fact that $\sigma$ has maximal dimension)}
						\State $l_{\sigma} \leftarrow$ linear function $\R^{n}\to \R$ with 
						\[l_{\sigma}(x_{1},\ldots,x_{n}) = \sum_{i=1}^{n}{\sigma_{i}\cdot x_{i}}\]
						\State $IsNew \leftarrow 1$
						\For{$(l, mt) \in L$}
							\If{$l_{\sigma} =_{2} l$}
								\State $IsNew \leftarrow 0$
								\If{$-\sigma_{P}>mt$}
									\State Replace $(l,mt)$ by $(l, -\sigma_{P})$ in $L$
								\EndIf
							\EndIf
						\EndFor
						\If{$IsNew=1$}
							\State Append $(l_{\sigma}, -\sigma_{P})$ to $L$
						\EndIf
					\EndIf
				\EndFor
			\EndFor
			\State \[g\leftarrow \bigodot_{(l,mt) \in L}{l^{\odot mt}}\]
		\Else
			\State $g\leftarrow 0$
		\EndIf
		\EndFunction
	\end{algorithmic}
\end{algorithm}

\begin{remark}
When $\S = \S_{K_n}$ or more generally $\S_\mathbb{G}$, there are two major computation shortcuts for Algorithms \ref{alg:gen-NS} and~\ref{alg:gen-ZS}. Firstly, as $\S$ is a full positive basis, it is easy to check if a polytope is a signed Minkowski sum of polytopes in $\S$. Secondly, the $H$-matrix is highly symmetric, and in this case, $b^H(P)$ for some polytope $P$ can be computed by M\"{o}bius inversion \cite[Proposition 2.4]{Ardila10}. 
\end{remark}

\section{Numerical Examples}\label{sec:example}

\begin{example}[Rational polynomials from spanning trees]\label{exa:k4-e}
This example is adapted from the family of $M$-convex functions given in \cite[Example 6.27]{Murota03}. Let $G$ be the edge-weighted graph on $n = 5$ edges shown below.
\begin{center}
\includegraphics[width=0.15\textwidth]{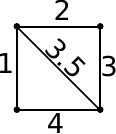}
\end{center}
Let $\Theta$ be the set of spanning trees of $G$. Define the following tropical polynomial $f_G: \R^n \to \R$
	\begin{equation}\label{eqn:tree}
	f_{G}(x)=\mathop{\bigoplus}_{T\in \Theta}{\left[ \left(-\bigodot_{e_{i}\in E(T)}{w_{i}}\right) \odot \bigodot_{e_{i}\in E(T)}{x_{i}} \right]}.
	\end{equation}
	Explicitly, $f_{G}(x)$ is the maximum of the following
	\begin{eqnarray*}
&x_1+x_2+x_3-6,x_1+x_2+x_4-7,x_1+x_3+x_4-8, \\ 
&x_2+x_3+x_4-9, x_1+x_2+x_5-6.5, x_1+x_3+x_5-7.5,\\
&x_2+x_4+x_5-9.5,x_3+x_4+x_5-10.5.
	\end{eqnarray*}
	
We find that $f_G$ is a unit, with
	\begin{equation}\label{eqn:7-4}
	\begin{split}
	& \Newt(f_{G})=\Delta_{\{1,4\}}+\Delta_{\{2,3\}}+\Delta_{\{1,2,5\}}+\Delta_{\{1,3,5\}}+\Delta_{\{2,4,5\}} \\
	& +\Delta_{\{3,4,5\}}-\Delta_{\{1,2,3,5\}}-\Delta_{\{1,2,4,5\}}-\Delta_{\{1,3,4,5\}}-\Delta_{\{2,3,4,5\}}+\Delta_{\{1,2,3,4,5\}}.
	\end{split}
	\end{equation}
	
	Algorithm \ref{alg:gen-ZS} outputs the following
	\[\begin{split}
	&g_G(x) = \max(x_{1},x_{2}-1,x_{3}-2,x_{5}-2.5) +\max(x_{1},x_{2}-1,x_{4}-3,x_{5}-2.5) \\
	&+\max(x_{1},x_{3}-2,x_{4}-3,x_{5}-2.5) +\max(x_{2},x_{3}-1,x_{4}-2,x_{5}-1.5).
	\end{split}
	\]
One can check that $f_{G}\odot g_{G}$ is still a unit, and $\Newt(f_{G}\odot g_{G})$ is a Minkowski sum of the $7$ simplices with positive coefficients amongst those in (\ref{eqn:7-4}). By Algorithm \ref{alg:gen-NS}, we get its factorization as follows:

	\begin{equation*}
	\begin{split}
f_{G}\odot g_{G}(x) 	& = \max(x_1,x_2-1,x_3-2,x_4-3,x_5-\frac{5}{2})\\
	& +\max(x_1,x_2-1,x_5-\frac{5}{2})+\max(x_1,x_3-2,x_5-\frac{5}{2})\\
	& +\max(x_2,x_4-2,x_5-\frac{3}{2})+\max(x_3,x_4-1,x_5-1/2)\\
	& +\max(x_1,x_4-3)+\max(x_2,x_3-1).
	\end{split}
	\end{equation*} 
\end{example}

Next we present a non-unit tropical polynomial $f\in \mathbb{Z}[\mathcal{S}_{K_3}]\backslash\mathbb{N}[\mathcal{S}_{K_3}]$. This example comes from a quadratic $M$-convex function in \cite[Example 2.10]{Murota03}.

\begin{example}\label{exa:quad}
	Let $f: \R^3 \to \R$ be a homogeneous quadric tropical polynomial that is the maximum of the following
	\begin{equation*}
	\begin{split}
& 3x_1-18, 3x_2-45, 3x_3-54, 3x_4-81, x_1+2x_2-34, x_1+2x_3-34, \\
	& x_1+2x_4-42, 2x_1+x_2-25, 2x_1+x_3-22, 2x_1+x_4-21, x_2+2x_3-45, \\
	& x_2+2x_4-53, 2x_2+x_3-42, 2x_2+x_4-41, x_3+2x_4-54, 2x_3+x_4-45, \\
	& x_1+x_2+x_3-31, x_1+x_2+x_4-30, x_1+x_3+x_4-29, x_2+x_3+x_4-40.
	\end{split}
	\end{equation*}
	
	The Newton polytope of $f$ is 3 times the standard simplex $\R^4$. Its regular subdivision consists of 14 maximal cells, which are all in $\Z\mathcal{S}_4$. Their signed Minkowski sum decompositions are
	\begin{equation*}
	\begin{split}
        & 1\cdot \Delta_{\{1\}} + 1\cdot \Delta_{\{1,2\}} + 1\cdot \Delta_{\{2,3,4\}},\\
		& 1\cdot \Delta_{\{1\}} + 1\cdot \Delta_{\{4\}} + 1\cdot \Delta_{\{1,2,3,4\}},\\
		& 1\cdot \Delta_{\{1,2\}} + 1\cdot \Delta_{\{2,3\}} + 1\cdot \Delta_{\{2,4\}} + 1\cdot \Delta_{\{3,4\}} -1\cdot \Delta_{\{2,3,4\}},\\
		& 1\cdot \Delta_{\{3\}} + 1\cdot \Delta_{\{1,2\}} + 1\cdot \Delta_{\{2,3,4\}},\\
		& 1\cdot \Delta_{\{3\}} + 1\cdot \Delta_{\{3,4\}} + 1\cdot \Delta_{\{1,2,3\}} + 1\cdot \Delta_{\{1,2,4\}} -1\cdot \Delta_{\{1,2,3,4\}},\\
		& 2\cdot \Delta_{\{3\}} + 1\cdot \Delta_{\{1,2,3,4\}},\\
		& 2\cdot \Delta_{\{4\}} + 1\cdot \Delta_{\{1,2,3,4\}},\\
		& 1\cdot \Delta_{\{3\}} + 1\cdot \Delta_{\{4\}} + 1\cdot \Delta_{\{1,2,3,4\}},\\
		& 1\cdot \Delta_{\{4\}} + 1\cdot \Delta_{\{3,4\}} + 1\cdot \Delta_{\{1,2,3\}} + 1\cdot \Delta_{\{1,2,4\}} -1\cdot \Delta_{\{1,2,3,4\}},\\
		& 1\cdot \Delta_{\{4\}} + 1\cdot \Delta_{\{1,2\}} + 1\cdot \Delta_{\{2,3,4\}},\\
		& 1\cdot \Delta_{\{2\}} + 1\cdot \Delta_{\{1,2\}} + 1\cdot \Delta_{\{2,3,4\}},\\
		& 1\cdot \Delta_{\{1\}} + 1\cdot \Delta_{\{3\}} + 1\cdot \Delta_{\{1,2,3,4\}},\\
		& 1\cdot \Delta_{\{1\}} + 1\cdot \Delta_{\{3,4\}} + 1\cdot \Delta_{\{1,2,3\}} + 1\cdot \Delta_{\{1,2,4\}} -1\cdot \Delta_{\{1,2,3,4\}},\\
		& 2\cdot \Delta_{\{1\}} + 1\cdot \Delta_{\{1,2,3,4\}}.
	\end{split}
	\end{equation*}
	
	There are four terms with negative coefficients. Algorithm \ref{alg:gen-ZS} outputs $g$ a product of four units
	
	\begin{equation*}
		\begin{split}
		g(x)& = \max(x_{2},x_{3}-11,x_{4}-10) +\max(x_{1},x_{2}-11,x_{3}-15,x_{4}-11)\\
		+&\max(x_{1},x_{2}-11,x_{3}-12,x_{4}-25) +\max(x_{1},x_{2}-9,x_{3}-8,x_{4}-7).
		\end{split}
	\end{equation*}
	
Algorithm \ref{alg:gen-NS} gives the following factorization 
	
	\begin{equation*}
		\begin{split}
		&(f \odot g)(x) \\
		&= \max(x_1,x_2-9,x_3-12,x_4-7)+ \max(x_1,x_2-9,x_3-8,x_4-21)\\
		& +\max(x_1,x_2-11,x_3-12,x_4-39)+ \max(x_1,x_2-11,x_3-16,x_4-25)\\
		& +\max(x_1,x_2-11,x_3-20,x_4-11)+ \max(x_1,x_2-7,x_3-4,x_4-3)\\
		& +\max(x_2-2,x_3-1,x_4).
		\end{split}
	\end{equation*} 
\end{example}

\section{Summary and open questions}\label{sec:summary}

In this work, we showed that if a finite set of lattice polytopes $\S$ is a positive basis, then one has an efficient algorithm to decide if a given tropical polynomial is $\S$-factorizable or strong $\S$-rational. Furthermore, when $\S$ is a full positive basis, then one has an even better description of strong $\S$-rationals. The tropical rational factorization solved in this paper is part of the recent efforts on generalizing Minkowski sum algorithms to signed Minkowski sums \cite{emiris2016efficient}. We close with a number of interesting open questions in polyhedral computations and tropical geometry.

A major open problem in our paper is to find full positive bases. It is easy to construct and verify positive bases, or construct a full basis. However, to construct one that is simultaneously full and positive is more difficult. It is not clear whether a full positive basis always exists for any set of primitive edge directions. We conjecture this to be true.
\begin{conjecture}[The Full Positive Basis Conjecture]\label{conj:full}
Let $E$ be a set of primitive lattice edges in $\R^n$. There is a full positive basis $\S$ such that $\S^1 = E$.
\end{conjecture}
If this conjecture holds, then for a given tropical polynomial $f$, let $e(f)$ be the set of primitive edges parallel to those in $\Delta_f$. Let $\S$ be a full positive basis such that $\S^1 = e(f)$. Then Theorem \ref{thm:SS} says that $f$ must be rationally factorizable, where both the numerator and denominator are $\S$-units. In other words, the conjecture implies that the following is true.
\begin{conjecture}[Conjecture for Rational Factorization]\label{cor:fundamental}
Any tropical polynomial in any number of variables is rationally factorizable into a product of affine monomials of the form
$$ (x_1,\dots,x_n) \mapsto \bigoplus_{i=0}^nc_i\oplus x^{\odot a_i}, $$
for some $a_i \in \mathbb{N}$
\end{conjecture}

It is easy to show that these two conjectures are true for $n = 2$ (cf. Proposition~\ref{prop:n2} and Corollary~\ref{cor:main.rational}). However, we do not know if either of them hold for $n \geq 3$. A sub-problem is the following the edge skeleton variant of the Minkowski reconstruction problem \cite{gritzmann1999algorithmic}.
\begin{open}
Characterize all lattice polytopes that can be constructed from a given set of edge directions. 
\end{open}
To the best of our knowledge, the solution for $n = 3$ is known \cite{martinez2019existence} but not in higher dimensions.  

Another major question is to generalize beyond the positive basis condition. As noted in Lemma \ref{lem:necessary}, a necessary condition for unique and local factorization is that $\S$ is a hierarchical basis. One could either ask for a tighter condition, that is, an if and only if characterization of unique and local factorization for $\NS$. Alternatively, one could sacrifice uniqueness and local factorization in favor of computational efficiency. This leads to the following concrete problem. 

\begin{open}
What other families of polytopes $\S$ where the membership problem for $\NS$ and $\ZS$ can be efficiently solved?
\end{open}

The above question can be taken further: if tropical factorization into units is a special case of the general tropical factorization problem of Speyer and Sturmfels \cite{Speyer09}, what are other special cases where irreducibility of tropical polynomials with respect to a given class can efficiently be determined?

\section*{Acknowledgment}\label{sec:ack}

The authors thank Bernd Sturmfels for suggesting the collaboration. They also thank Elizabeth Baldwin and Paul Klemperer for stimulating discussions and Kazuo Murota for enlightening them the intricate and interesting properties of the $M$-convex functions. Bo Lin is supported by the Max-Planck Institute of Mathematics in the Sciences (Leipzig, Germany) and the 2017 summer grant for doctoral students provided by the Graduate Division of University of California, Berkeley. Ngoc Mai Tran is supported by the Bonn Junior Fellowship of the Hausdorff Center for Mathematics. 

\nocite{*}
\bibliographystyle{alpha}
\newcommand{\etalchar}[1]{$^{#1}$}

\end{document}